\documentclass[12pt,reqno,tbtags]{amsart}
\usepackage[utf8]{inputenc}
\usepackage[T1]{fontenc}

\usepackage{amssymb}
\usepackage{amsthm}
\usepackage{amsmath, bm}
\usepackage{bbm}
\usepackage{mathrsfs}
\usepackage{esint} 
\usepackage{graphicx}
\usepackage{geometry}
\usepackage{hyperref}
\usepackage{float}
\usepackage{cite} 
\usepackage{esint}
\usepackage{pgfplots}
\usepackage{tikz-cd}
\usepackage{mdframed}
\usetikzlibrary{intersections}

\newtheorem{theorem}{Theorem}[section]

\newtheorem{definition}{Definition}[section]

\newtheorem{lemma}{Lemma}[section]
\newtheorem{remark}{Remark}[section]
\newtheorem{corollary}{Corollary}[section]

\allowdisplaybreaks[4]

\def\R{\mathbb{R}}

\begin{document}

\title[The Cauchy problem for the improved Boussinesq equation with spatially quasi-periodic initial data]{The Cauchy problem for the improved Boussinesq equation with spatially quasi-periodic initial data}

\author{Zhiqiang Wan}
\thanks{}
\address{School of Mathematical Sciences, University of Science and Technology of China, No. 96 Jinzhai Road, Baohe District, Hefei, Anhui Province, China}
\email{ZhiQiang\_Wan576@mail.ustc.edu.cn}

\author{Wenji Wu}
\thanks{}
\address{School of Mathematical Sciences, University of Science and Technology of China, No. 96 Jinzhai Road, Baohe District, Hefei, Anhui Province, China}
\email{wuwenji@mail.ustc.edu.cn}

\author{Heng Zhang}
\thanks{}
\address{School of Mathematical Sciences, University of Science and Technology of China, No. 96 Jinzhai Road, Baohe District, Hefei, Anhui Province, China}
\email{hengz@mail.ustc.edu.cn}

\date{\today}

\begin{abstract}
We study the Cauchy problem for the improved Boussinesq equation
\[
    u_{tt}-u_{xx}-u_{xxtt}-(u^2)_{xx}=0
\]
on the real line with spatially quasi-periodic initial data. For a non-resonant
frequency vector $\omega\in\mathbb R^\nu$, we prove local existence and uniqueness of classical spatially
quasi-periodic solutions with the same frequency vector $\omega$ in two
Fourier-side classes. First, for exponentially decaying initial Fourier
coefficients, we obtain a spatially quasi-periodic solution whose Fourier
coefficients remain exponentially decaying on an explicit time interval. Second,
for initial Fourier coefficients $c(n)$ and $d(n)$ satisfying the polynomial
decay
$
    |c(n)|+|d(n)|\lesssim (1+|n|)^{-r},
    \; r>\nu+2,
$
we prove that the corresponding spatially quasi-periodic solution preserves the
same polynomial decay rate as the initial data. We also extend these results to the nonlinearity $u^p$ with integer $p \geq 3$.
\end{abstract}

\maketitle

\section{Introduction}

In recent years, the initial value problem for time-dependent partial differential equations with almost periodic data has attracted growing attention. A major motivation for this line of research is the so-called Deift conjecture, which asserts that, for the KdV equation, if the initial data is almost periodic, then the corresponding solution should evolve almost periodically in time \cite{Deift17}. This conjectural picture has been examined in a variety of settings, including counterexamples \cite{CKV24,DLVY21}, almost periodic initial data \cite{BDGL18,EVY19,LY20}, the Toda lattice \cite{BDLV18,LYZZ24,ZC25},  quasi-periodic initial data for the KdV equation \cite{DG16} and the generalized KdV equation \cite{DLX24}, and other dispersive equations \cite{DLX24:JFA, Papenburg25, Schippa25,Xu25}.

Motivated by this circle of ideas, in this paper we consider the Cauchy problem for the improved Boussinesq equation with quasi-periodic initial data on $\mathbb{R}$, namely
\begin{equation}\label{eq:IBq-intro}
    u_{tt}-u_{xx}-u_{xxtt}-(u^2)_{xx}=0,
    \qquad (t,x)\in \mathbb{R}\times\mathbb{R},
\end{equation}
together with initial data of the form
\begin{equation}\label{eq:IBq-data-intro}
    u(0,x)=\sum_{n\in\mathbb{Z}^{\nu}}c(n)e^{i\langle n,\omega\rangle x},
    \qquad
    u_t(0,x)=\sum_{n\in\mathbb{Z}^{\nu}}d(n)e^{i\langle n,\omega\rangle x},
\end{equation}
Here the frequency dimension satisfies $\nu\ge 2$ with $\nu\in\mathbb{N}$. The vector
$
\omega=(\omega_1,\dots,\omega_\nu)\in\mathbb{R}^\nu
$
is a prescribed frequency vector, while
$
n=(n_1,\dots,n_\nu)\in\mathbb{Z}^\nu
$
denotes a lattice vector. We also write 
$
\langle n,\omega\rangle:=\sum_{j=1}^{\nu} n_j\omega_j.
$
As usual, we assume that the frequency vector $\omega$ is non-resonant, or equivalently rationally independent, that is,
\[
\langle n,\omega\rangle=0 \quad \Longrightarrow \quad n=0\in\mathbb{Z}^\nu.
\] We are interested in spatially quasi-periodic solutions of the form
\begin{equation}\label{eq:IBq-solution-intro}
    u(t,x)=\sum_{n\in\mathbb{Z}^{\nu}}\widehat{u}(t,n)e^{i\langle n,\omega\rangle x}
\end{equation}
to the quasi-periodic Cauchy problem \eqref{eq:IBq-intro}--\eqref{eq:IBq-data-intro} in the classical sense.

The improved Boussinesq equation belongs to the Boussinesq family of dispersive equations arising as asymptotic models for the bidirectional propagation of small-amplitude long surface waves in shallow water \cite{BS76,BCS02,BCS04}. From the analytical point of view, various aspects of the Cauchy problem for improved and generalized Boussinesq-type equations have been studied extensively, including well-posedness, small-data global existence, and scattering behavior; see, for example, \cite{WC02a,Wang09, WC02b,CO06}. In the quasi-periodic setting, however, the available results remain rather limited. Let us also mention the recent work by Gao-Li-Su\cite{GLS23} on the good Boussinesq equation with quasi-periodic initial data.

The spatial dependence in the present paper is neither decaying nor periodic, and this makes the problem considerably more delicate. Indeed, once one inserts the quasi-periodic Fourier series into \eqref{eq:IBq-intro}, the equation is reduced to an infinite nonlinear system of coupled ordinary differential equations for the Fourier coefficients. The corresponding Picard iteration then involves repeated discrete convolution on $\mathbb{Z}^{\nu}$, and the number of terms grows rapidly as the iteration proceeds. Thus, in order to prove convergence of the Picard sequence, one must exploit the precise combinatorial structure generated by the nonlinearity.

Throughout this paper, we use $|n|$ to denote the $\ell^1(\mathbb{Z}^\nu)$-norm of
$n=(n_1,\dots,n_\nu)\in\mathbb{Z}^\nu$, that is,
\[
|n|:=\sum_{j=1}^{\nu}|n_j|.
\]

Our first main result is stated in Theorem \ref{thm:exponential decay}, under the assumption of exponential decay.

\begin{theorem}\label{thm:exponential decay}
Let $\omega\in \mathbb{R}^{\nu}$ be non-resonant, and consider the quasi-periodic
Cauchy problem for \eqref{eq:IBq-intro}
with initial data \eqref{eq:IBq-data-intro}.
Assume that there exist $A>0$ and $0<\rho\leq 1$ such that
\[
|c(n)|\leq A e^{-\rho |n|},
\qquad
|d(n)|\leq A e^{-\rho |n|},
\qquad n\in \mathbb{Z}^{\nu}.
\]
Set
\[
b_{\rho}:=(6\rho^{-1})^{\nu},
\qquad
M:=\max\{1,Ab_{\rho}\},
\qquad
B:=2M,
\qquad
L:=\frac{1}{5M}.
\]
Then, on the time interval $[0,L]$, the quasi-periodic Cauchy problem admits a classical spatially
quasi-periodic solution of the form
\[
u(t,x)=\sum_{n\in \mathbb{Z}^{\nu}} \widehat{u}(t,n)e^{i\langle n,\omega\rangle x}
\]
such that
\[
|\widehat{u}(t,n)|\leq B e^{-\frac{\rho}{2}|n|},
\qquad
0\leq t\leq L,\ \ n\in \mathbb{Z}^{\nu}.
\]
Moreover, this solution is unique among all spatially quasi-periodic solutions
satisfying the above exponential decay estimate.
\end{theorem}


We now turn to polynomially decaying initial Fourier coefficients. In contrast with
many related quasi-periodic results, the decay rate in our polynomial theorem is
preserved by the nonlinear evolution. This is our second main result.

\begin{theorem}\label{thm:polynomial decay}
Let $\omega\in \mathbb{R}^{\nu}$ be non-resonant, and consider the quasi-periodic
Cauchy problem of \eqref{eq:IBq-intro}
with initial data \eqref{eq:IBq-data-intro}.
Assume that there exist $A>0$ and $r>\nu+2$ such that
\[
|c(n)|\le A(1+|n|)^{-r},
\quad
|d(n)|\le A(1+|n|)^{-r},
\quad n\in\mathbb Z^\nu.
\]
Let
\[
H(r;\nu):=\sum_{m\in\mathbb Z^\nu}(1+|m|)^{-r},
\quad
K_{r,\nu}:=2^{r+1}H(r;\nu),
\]
and set
\[
M_r:=\max\{1,A K_{r,\nu}\}, \quad
L_r:=\frac{1}{5M_r}.
\]
Then, on the time interval $[0,L_r]$, the quasi-periodic Cauchy problem admits a classical spatially
quasi-periodic solution of the form
\[
u(t,x)=\sum_{n\in\mathbb Z^\nu}\widehat u(t,n)e^{i\langle n,\omega\rangle x}
\]
such that
\[
|\widehat u(t,n)|\le 2A(1+|n|)^{-r},
\qquad
0\le t\le L_r,\ \ n\in\mathbb Z^\nu.
\]
Moreover, this solution is unique among all spatially quasi-periodic solutions
satisfying the above polynomial decay estimate.
\end{theorem}

\begin{remark} In the polynomial-decay case, we only require the condition 
$r>\nu+2$, which is precisely the minimal assumption needed to justify the absolute and uniform convergence of the Fourier series after term-by-term differentiation applied to the equation. More importantly, the solution inherits the polynomial decay rate of the initial data, so that no decay loss occurs throughout the iteration procedure. Our approach is comparable to the recent work of Damanik, Li and Xu in \cite{DLX24:JFA} on the generalized Benjamin--Bona--Mahony equation. More precisely, while the polynomial result in \cite{DLX24:JFA} assumes $r>4\nu+8$, and yields only the decay
$(1+|n|)^{-r/2}$ for the solution, Theorem~\ref{thm:polynomial decay}
requires only $r>\nu+2$ and preserves the full decay rate
$(1+|n|)^{-r}$.
\end{remark}

Let us emphasize that, in the integrable KdV/Toda setting, spectral information
plays a crucial role in obtaining global-in-time results for quasi-periodic or almost
periodic data. For the KdV equation, Damanik and Goldstein~\cite{DG16} first
established a local existence and uniqueness result for quasi-periodic initial data
with exponentially decaying Fourier coefficients, and then, for small analytic
quasi-periodic data with a Diophantine frequency vector, used the conservation of
the spectrum of the associated Sturm--Liouville operator, together with inverse
spectral estimates for quasi-periodic Schr\"odinger operators, to iterate the local
theory and obtain global existence and uniqueness. More recently,
Leguil--You--Zhao--Zhou~\cite{LYZZ24} obtained sharp information on spectral gaps
and homogeneity of spectra for quasi-periodic Schr\"odinger operators, and used
these spectral results to prove a discrete version of Deift's conjecture for the
Toda flow, namely global solutions and almost periodicity in time for
subcritical analytic quasi-periodic initial data. In contrast, for the Boussinesq
equation considered here, no comparable inverse spectral mechanism is available;
accordingly, our approach is purely Fourier-analytic and combinatorial, and yields
local existence and uniqueness.

The proofs of Theorem~\ref{thm:exponential decay} and Theorem~\ref{thm:polynomial decay} follow a common
strategy, see Figure \ref{fig:proof-strategy}. We first reduce the PDE to an infinite nonlinear system of coupled ordinary
differential equations for the Fourier coefficients. We then construct a Picard
sequence and represent it by means of a combinatorial tree adapted to the present
second-order evolution, in which the initial position and initial velocity generate
two different types of leaves. In the exponential-decay setting, the main task is to
derive uniform exponential bounds for the Picard sequence and prove its convergence.
In the polynomial-decay setting, the crucial new ingredient is a weighted convolution
estimate in a suitable Fourier-side space, which allows us to close the nonlinear
iteration without sacrificing the decay exponent. Uniqueness is then obtained by a
Gronwall-type argument in the corresponding weighted norm. 

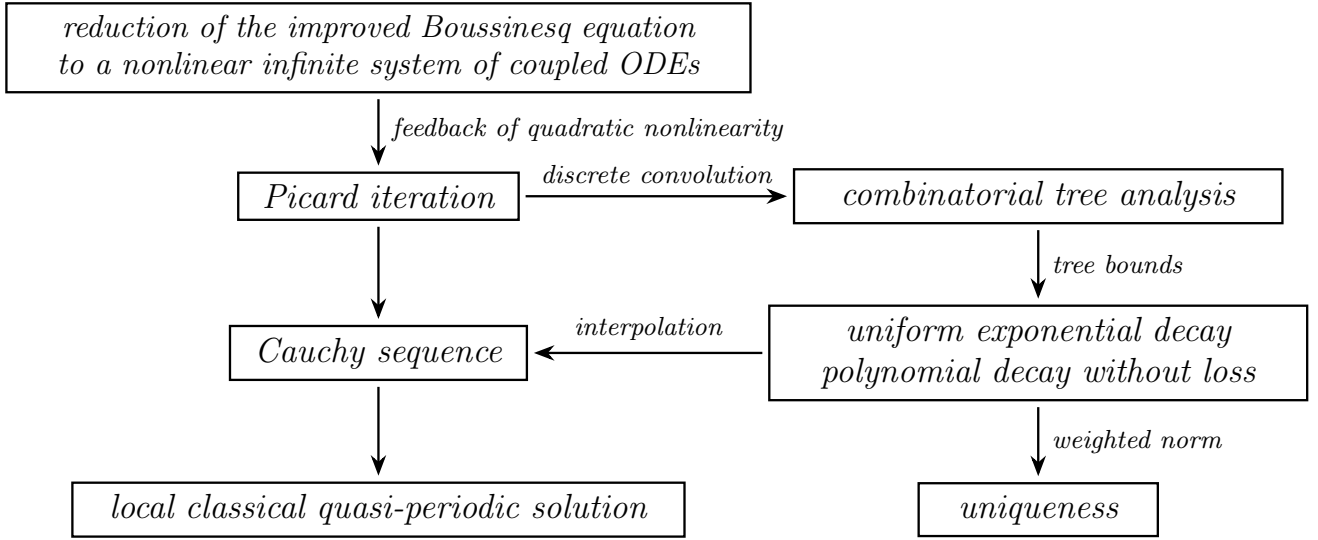
\begin{figure}[htbp]
\centering
\resizebox{0.98\textwidth}{!}{%
\begin{tikzpicture}[
    x=1cm,y=1cm,
    box/.style={
        draw,
        thick,
        rectangle,
        align=center,
        inner xsep=6pt,
        inner ysep=4pt,
        font=\normalsize\itshape
    },
    topbox/.style={
        box,
        font=\large\itshape
    },
    arr/.style={
        -{Stealth[length=2.5mm]},
        thick,
        shorten >=2pt,
        shorten <=2pt
    },
    lab/.style={
        fill=white,
        inner sep=1.5pt,
        align=center,
        font=\scriptsize\itshape
    }
]

\node[box, font=\small\itshape, text width=8.6cm] (reduction) at (-3.0,6.45)
{reduction of the improved Boussinesq equation\\
to a nonlinear infinite system of coupled ODEs};

\node[box, minimum width=3.4cm] (picard) at (-3.0,4.65)
{Picard iteration};

\node[box, minimum width=3.6cm] (cauchy) at (-3.0,2.75)
{Cauchy sequence};

\node[box, text width=7.0cm] (existence) at (-3.0,0.85)
{local classical quasi-periodic solution};

\node[box, text width=5.5cm] (tree) at (5.0,4.65)
{combinatorial tree analysis};

\node[box, text width=6.1cm] (decay) at (5.0,2.75)
{uniform exponential decay\\
 polynomial decay without loss};

\node[box, minimum width=2.9cm] (uniq) at (5.0,0.85)
{uniqueness};

\draw[arr] (reduction.south) -- node[lab, right=1mm, pos=.46]
{feedback of quadratic nonlinearity} (picard.north);

\draw[arr] (picard.south) -- (cauchy.north);

\draw[arr] (cauchy.south) -- (existence.north);

\draw[arr] (picard.east) -- node[lab, above=2.5pt]
{discrete convolution} (tree.west);

\draw[arr] (tree.south) -- node[lab, right=1mm]
{tree bounds} (decay.north);

\draw[arr] (decay.west) -- node[lab, above=2.5pt]
{interpolation} (cauchy.east);

\draw[arr] (decay.south) -- node[lab, right=1mm]
{weighted norm} (uniq.north);

\end{tikzpicture}%
}
\caption{Structure of the proof.}
\label{fig:proof-strategy}
\end{figure}

We remark that the quadratic nonlinearity in \eqref{eq:IBq-intro} is not essential
for the local theory developed in this paper. The same argument applies to the
generalized improved Boussinesq equation
\[
    u_{tt}-u_{xx}-u_{xxtt}-(u^p)_{xx}=0,
    \qquad p\in\mathbb N,\quad p\ge3.
\]
Indeed, compared with the quadratic case, the only structural change is the replacement of
the binary convolution by the $p$-fold convolution. In a full tree expansion, this
corresponds to replacing binary trees by $p$-ary trees. The uniform boundedness of the
Boussinesq multiplier, together with the weighted convolution estimates used below,
allows one to close the Picard iteration in both the exponential-decay and
polynomial-decay classes. Thus analogous local existence and uniqueness results hold
for all integer powers $p\ge3$. We give a sketch of proof of  this extension in
Section~\ref{sec:general-power}.

We also mention that although the present paper is concerned with the one-dimensional spatial problem,
the combinatorial mechanism used here is not intrinsically restricted to one space
dimension. What is essential is the reduction of the quasi-periodic Cauchy problem
to an infinite system of coupled ODEs in Fourier space, together with a precise
description of the discrete convolutions generated by the Picard iteration. For
higher-dimensional spatial variables with a separated quasi-periodic structure, for
instance with a diagonal frequency matrix,
the same philosophy can still be implemented, although the bookkeeping of lattice
variables becomes more involved. In this direction, Xu~\cite{Xu25} developed a
diagrammatic version of the combinatorial analysis to treat higher-dimensional
weakly nonlinear Schr\"odinger equations with quasi-periodic initial data.

The rest of this paper is organized as follows. In Section \ref{sec:2}, we reduce the
quasi-periodic Cauchy problem to an infinite nonlinear system of coupled ordinary
differential equations for the Fourier coefficients and construct the associated
Picard sequence. We then introduce the combinatorial tree expansion adapted to the
second-order structure of the equation and use it to prove the required uniform
decay estimates, convergence of the Picard sequence, and uniqueness. This proves
Theorem~\ref{thm:exponential decay}. In Section \ref{sec:3}, we work in a weighted Fourier-side space
and establish the convolution estimate needed to close the iteration in the
polynomial-decay class without any loss of decay, thereby proving
Theorem~\ref{thm:polynomial decay}. Finally, in Section~\ref{sec:general-power},
we explain how the same method extends to generalized improved Boussinesq equations
with integer-power nonlinearities $u^p$, $p\geq 3$.

\section{Picard sequence and proof of Theorem \ref{thm:exponential decay}}\label{sec:2}

\subsection{Reduction}\label{sec:reduction}

In this subsection, we reduce the quasi-periodic Cauchy problem for
\eqref{eq:IBq-intro}
to an infinite nonlinear system of coupled ordinary differential equations in Fourier space.

Let $\omega\in\mathbb{R}^{\nu}$ be a fixed non-resonant frequency vector, and consider a spatially quasi-periodic solution of the form
\begin{equation*}
u(t,x)=\sum_{n\in\mathbb{Z}^{\nu}} c(t,n)e^{i\langle n,\omega\rangle x}.
\end{equation*}
The initial data are,
\begin{equation*}
u(0,x)=\sum_{n\in\mathbb{Z}^{\nu}} c(n)e^{i\langle n,\omega\rangle x},
\qquad
u_t(0,x)=\sum_{n\in\mathbb{Z}^{\nu}} d(n)e^{i\langle n,\omega\rangle x},
\end{equation*}
where
\[
c(n)=c(0,n), \qquad d(n)=\partial_t c(0,n).
\]

For convenience, write
\[
\theta_n:=\langle n,\omega\rangle, \qquad n\in\mathbb{Z}^{\nu}.
\]
Formally, by the Cauchy product for infinite series, we have
\begin{equation*}
u^2(t,x)
=
\sum_{n\in\mathbb{Z}^{\nu}}
\left(
\sum_{\substack{n_1,n_2\in\mathbb{Z}^{\nu}\\ n_1+n_2=n}}
c(t,n_1)c(t,n_2)
\right)e^{i\theta_n x}.
\end{equation*}
Assuming that differentiation and summation can be interchanged, we obtain
\begin{align}
u_{tt}(t,x)
&=
\sum_{n\in\mathbb{Z}^{\nu}} c_{tt}(t,n)e^{i\theta_n x}, \label{eq:utt}\\
-u_{xx}(t,x)
&=
\sum_{n\in\mathbb{Z}^{\nu}} \theta_n^2 c(t,n)e^{i\theta_n x}, \notag\\
-u_{xxtt}(t,x)
&=
\sum_{n\in\mathbb{Z}^{\nu}} \theta_n^2 c_{tt}(t,n)e^{i\theta_n x}, \notag\\
-(u^2)_{xx}(t,x)
&=
\sum_{n\in\mathbb{Z}^{\nu}}
\theta_n^2
\left(
\sum_{\substack{n_1,n_2\in\mathbb{Z}^{\nu}\\ n_1+n_2=n}}
c(t,n_1)c(t,n_2)
\right)e^{i\theta_n x}. \label{eq:nonlinear-term}
\end{align}
Substituting \eqref{eq:utt}--\eqref{eq:nonlinear-term} into \eqref{eq:IBq-intro}, we get
\[
\sum_{n\in\mathbb{Z}^{\nu}}
\Bigg[
(1+\theta_n^2)c_{tt}(t,n)
+\theta_n^2 c(t,n)
+\theta_n^2
\sum_{\substack{n_1,n_2\in\mathbb{Z}^{\nu}\\ n_1+n_2=n}}
c(t,n_1)c(t,n_2)
\Bigg]e^{i\theta_n x}
=0.
\]
Interpreting Fourier coefficients with respect to the mean inner product introduced below, \eqref{eq:IBq-intro} is therefore equivalent to the following infinite system of coupled ODEs:
\begin{equation}\label{eq:ode-system-raw}
(1+\theta_n^2)c_{tt}(t,n)
+\theta_n^2 c(t,n)
+\theta_n^2
\sum_{\substack{n_1,n_2\in\mathbb{Z}^{\nu}\\ n_1+n_2=n}}
c(t,n_1)c(t,n_2)
=0,
\qquad n\in\mathbb{Z}^{\nu}.
\end{equation}

It is convenient to introduce
\begin{equation*}
\beta(n):=\frac{\theta_n^2}{1+\theta_n^2},
\qquad
\Omega(n):=\sqrt{\beta(n)}=\frac{|\theta_n|}{\sqrt{1+\theta_n^2}},
\qquad n\in\mathbb{Z}^{\nu}.
\end{equation*}
Then \eqref{eq:ode-system-raw} can be rewritten as
\begin{equation}\label{eq:ode-system}
c_{tt}(t,n)+\Omega(n)^2 c(t,n)
=
-\Omega(n)^2
\sum_{\substack{n_1,n_2\in\mathbb{Z}^{\nu}\\ n_1+n_2=n}}
c(t,n_1)c(t,n_2),
\qquad n\in\mathbb{Z}^{\nu}.
\end{equation}
Observe that
\[
0\leq \beta(n)<1,
\qquad
0\leq \Omega(n)<1,
\qquad n\in\mathbb{Z}^{\nu}.
\]

Next we derive the associated integral equation. For $n\neq 0$, motivated by an idea from the \cite{KPV91}, by variation of constants formula applied to \eqref{eq:ode-system}, we have
\begin{align}
c(t,n)
&=
\cos(\Omega(n)t)c(n)
+\frac{\sin(\Omega(n)t)}{\Omega(n)}\,d(n) \notag\\
&\quad
-\Omega(n)^2
\int_0^t
\frac{\sin(\Omega(n)(t-\tau))}{\Omega(n)}
\sum_{\substack{n_1,n_2\in\mathbb{Z}^{\nu}\\ n_1+n_2=n}}
c(\tau,n_1)c(\tau,n_2)\,d\tau \notag\\
&=
\cos(\Omega(n)t)c(n)
+\frac{\sin(\Omega(n)t)}{\Omega(n)}\,d(n) \notag\\
&\quad
-\Omega(n)
\int_0^t
\sin(\Omega(n)(t-\tau))
\sum_{\substack{n_1,n_2\in\mathbb{Z}^{\nu}\\ n_1+n_2=n}}
c(\tau,n_1)c(\tau,n_2)\,d\tau .
\label{eq:integral-nonzero}
\end{align}
For $n=0$, since $\theta_0=0$, \eqref{eq:ode-system} reduces to
\[
c_{tt}(t,0)=0,
\]
and hence
\begin{equation}\label{eq:integral-zero}
c(t,0)=c(0)+t\,d(0).
\end{equation}

To unify \eqref{eq:integral-nonzero} and \eqref{eq:integral-zero}, we define
\begin{equation*}
G_n(t):=
\begin{cases}
\cos(\Omega(n)t), & n\neq 0,\\
1, & n=0,
\end{cases}
\qquad
K_n(t):=
\begin{cases}
\dfrac{\sin(\Omega(n)t)}{\Omega(n)}, & n\neq 0,\\
t, & n=0.
\end{cases}
\end{equation*}
Then \eqref{eq:ode-system} is equivalent to
\begin{equation}\label{eq:integral-equation}
c(t,n)
=
G_n(t)c(n)+K_n(t)d(n)
-\beta(n)\int_0^t K_n(t-\tau)
\sum_{\substack{n_1,n_2\in\mathbb{Z}^{\nu}\\ n_1+n_2=n}}
c(\tau,n_1)c(\tau,n_2)\,d\tau,
\qquad n\in\mathbb{Z}^{\nu}.
\end{equation}

Motivated by \eqref{eq:integral-equation}, we define the Picard sequence $\{c_k(t,n)\}_{k\geq 0}$ by taking
\begin{equation*}
c_0(t,n):=G_n(t)c(n)+K_n(t)d(n),
\end{equation*}
and, for every $k\geq 1$,
\begin{equation}\label{eq:picard-iteration-recall}
c_k(t,n)
:=
c_0(t,n)
-\beta(n)\int_0^t K_n(t-\tau)
\sum_{\substack{n_1,n_2\in\mathbb{Z}^{\nu}\\ n_1+n_2=n}}
c_{k-1}(\tau,n_1)c_{k-1}(\tau,n_2)\,d\tau.
\end{equation}
Our goal in the subsequent subsections is to show that the Picard sequence converges and hence yields a spatially quasi-periodic solution to \eqref{eq:IBq-intro}.

\subsection{Combinatorial tree and flattening of frequency variables}\label{sec:tree-setup}

In this subsection, we introduce the combinatorial objects associated with the Picard
iteration. The tree representation itself will be stated in the next subsection.

We emphasize that, at this stage, all manipulations are understood formally.
The interchange of summation and integration, as well as the rearrangement of the
series below, will be justified later by the absolute convergence estimates proved in
Subsection~\ref{sec:uniform-decay}. We also recall that the identification of Fourier
coefficients is understood with respect to the mean inner product
\[
\langle f,g\rangle_{\mathrm{av}}
:=
\lim_{L\to\infty}\frac{1}{2L}\int_{-L}^{L} f(x)\overline{g(x)}\,dx .
\]

For the second-order equation considered here, it is convenient to merge the time
propagator and the nonlinear weight into a single factor. More precisely, define
\begin{equation}\label{eq:Phi-definition}
\Phi_n(t):=\beta(n)K_n(t)
=
\begin{cases}
\Omega(n)\sin(\Omega(n)t), & n\neq 0,\\[1mm]
0, & n=0.
\end{cases}
\end{equation}

The linear part of the solution contains two components, namely the contribution
generated by the initial position $c(n)$ and the one generated by the initial
velocity $d(n)$. Accordingly, our tree has two kinds of leaves. Set
\[
\mathfrak{L}:=\{\mathtt{c},\mathtt{d}\},
\]
where $\mathtt{c}$ corresponds to the leaf associated with $c(n)$, and $\mathtt{d}$
corresponds to the leaf associated with $d(n)$.

We define the set of trees recursively by
\begin{equation*}
\mathfrak{T}(0):=\mathfrak{L},
\qquad
\mathfrak{T}(k):=\mathfrak{L}\cup \big(\mathfrak{T}(k-1)\big)^2,
\qquad k\geq 1.
\end{equation*}
Thus, for every $k\geq 0$, the set $\mathfrak{T}(k)$ contains the two degenerate
trees $\mathtt{c}$ and $\mathtt{d}$, as well as all binary trees generated by the
quadratic nonlinearity up to depth $k$.

For every $\gamma\in \bigcup_{k\geq 0}\mathfrak{T}(k)$, define recursively
\begin{equation}\label{eq:sigma-definition}
\sigma(\gamma):=
\begin{cases}
1, & \gamma=\mathtt{c}\ \text{or}\ \gamma=\mathtt{d},\\[1mm]
\sigma(\gamma_1)+\sigma(\gamma_2), & \gamma=(\gamma_1,\gamma_2).
\end{cases}
\end{equation}
Hence $\sigma(\gamma)$ is the number of leaves of $\gamma$.

Next we define the frequency sets attached to each tree. For $k\geq 0$, let
\[
N(k,\gamma):=
\begin{cases}
\mathbb{Z}^{\nu}, & \gamma\in \mathfrak{L},\\[2mm]
N(k-1,\gamma_1)\times N(k-1,\gamma_2), &
k\geq 1,\ \gamma=(\gamma_1,\gamma_2)\in \big(\mathfrak{T}(k-1)\big)^2 .
\end{cases}
\]
For every $m\in N(k,\gamma)$, we define its total frequency recursively by
\begin{equation}\label{eq:mu-definition}
\mu(m):=
\begin{cases}
m, & m\in \mathbb{Z}^{\nu},\\[1mm]
\mu(m_1)+\mu(m_2), & m=(m_1,m_2).
\end{cases}
\end{equation}

To make the leaf variables explicit, we now introduce a flattening map and a leaf-label
sequence.

For each $k\geq 0$ and each $\gamma\in\mathfrak{T}(k)$, define recursively a map
\[
\mathrm{Fl}_{k,\gamma}:N(k,\gamma)\to (\mathbb{Z}^{\nu})^{\sigma(\gamma)}
\]
and a sequence
\[
\eta_{k,\gamma}
=
\bigl(\eta_{k,\gamma,1},\dots,\eta_{k,\gamma,\sigma(\gamma)}\bigr)
\in \{\mathtt{c},\mathtt{d}\}^{\sigma(\gamma)}
\]
as follows.

\medskip
\noindent
\textbf{Leaf case.}
If $\gamma=\mathtt{c}$, then
\[
\mathrm{Fl}_{k,\mathtt{c}}(m):=(m),
\qquad
\eta_{k,\mathtt{c}}:=(\mathtt{c}).
\]
If $\gamma=\mathtt{d}$, then
\[
\mathrm{Fl}_{k,\mathtt{d}}(m):=(m),
\qquad
\eta_{k,\mathtt{d}}:=(\mathtt{d}).
\]

\medskip
\noindent
\textbf{Non-leaf case.}
If $k\geq 1$ and $\gamma=(\gamma_1,\gamma_2)\in (\mathfrak{T}(k-1))^2$, then for
$m=(m^{(1)},m^{(2)})\in N(k,\gamma)=N(k-1,\gamma_1)\times N(k-1,\gamma_2)$, define
\[
\mathrm{Fl}_{k,\gamma}(m)
:=
\bigl(\mathrm{Fl}_{k-1,\gamma_1}(m^{(1)}),\,
\mathrm{Fl}_{k-1,\gamma_2}(m^{(2)})\bigr),
\]
where the right-hand side means concatenation of the two tuples, and define
\[
\eta_{k,\gamma}
:=
\bigl(\eta_{k-1,\gamma_1},\,\eta_{k-1,\gamma_2}\bigr),
\]
again by concatenation.

\begin{lemma}\label{lem:flattening}
For every $k\geq 0$ and every $\gamma\in\mathfrak{T}(k)$, the following statements hold.

\begin{enumerate}
\item[(1)] The map $\mathrm{Fl}_{k,\gamma}$ is a bijection from $N(k,\gamma)$ onto
$(\mathbb{Z}^{\nu})^{\sigma(\gamma)}$.

\item[(2)] If
\[
\mathrm{Fl}_{k,\gamma}(m)=(m_1,\dots,m_{\sigma(\gamma)}),
\]
then
\begin{equation}\label{eq:mu-flattening}
\mu(m)=\sum_{j=1}^{\sigma(\gamma)} m_j.
\end{equation}
\end{enumerate}
\end{lemma}

\begin{proof}
We argue by induction on the tree structure.

If $\gamma=\mathtt{c}$ or $\gamma=\mathtt{d}$, then
\[
N(k,\gamma)=\mathbb{Z}^{\nu},
\qquad
\sigma(\gamma)=1,
\]
and $\mathrm{Fl}_{k,\gamma}(m)=(m)$. Hence $\mathrm{Fl}_{k,\gamma}$ is a bijection
from $\mathbb{Z}^{\nu}$ onto $(\mathbb{Z}^{\nu})^1$, and
\[
\mu(m)=m=\sum_{j=1}^{1}m_j.
\]

Now let $k\geq 1$ and $\gamma=(\gamma_1,\gamma_2)\in(\mathfrak{T}(k-1))^2$.
Let
\[
m=(m^{(1)},m^{(2)})\in N(k,\gamma)
=
N(k-1,\gamma_1)\times N(k-1,\gamma_2).
\]
Write
\[
\mathrm{Fl}_{k-1,\gamma_1}(m^{(1)})
=
(m_1,\dots,m_{\sigma(\gamma_1)}),
\]
and
\[
\mathrm{Fl}_{k-1,\gamma_2}(m^{(2)})
=
(m_{\sigma(\gamma_1)+1},\dots,m_{\sigma(\gamma_1)+\sigma(\gamma_2)}).
\]
Since
\[
\sigma(\gamma)=\sigma(\gamma_1)+\sigma(\gamma_2),
\]
the recursive definition yields
\[
\mathrm{Fl}_{k,\gamma}(m)
=
(m_1,\dots,m_{\sigma(\gamma)}).
\]
By the induction hypothesis, both
\[
\mathrm{Fl}_{k-1,\gamma_1}:N(k-1,\gamma_1)\to(\mathbb{Z}^{\nu})^{\sigma(\gamma_1)}
\]
and
\[
\mathrm{Fl}_{k-1,\gamma_2}:N(k-1,\gamma_2)\to(\mathbb{Z}^{\nu})^{\sigma(\gamma_2)}
\]
are bijections. Hence $\mathrm{Fl}_{k,\gamma}$ is also a bijection onto
$(\mathbb{Z}^{\nu})^{\sigma(\gamma)}$, proving (1).

Moreover, by \eqref{eq:mu-definition},
\[
\mu(m)=\mu(m^{(1)})+\mu(m^{(2)}).
\]
Applying the induction hypothesis to $m^{(1)}$ and $m^{(2)}$, we obtain
\[
\mu(m^{(1)})=\sum_{j=1}^{\sigma(\gamma_1)}m_j,
\qquad
\mu(m^{(2)})=\sum_{j=\sigma(\gamma_1)+1}^{\sigma(\gamma)}m_j.
\]
Therefore,
\[
\mu(m)=\sum_{j=1}^{\sigma(\gamma)}m_j,
\]
which proves (2).
\end{proof}

Set
\[
a_{\mathtt c}(n):=c(n),
\qquad
a_{\mathtt d}(n):=d(n),
\qquad n\in\mathbb Z^\nu.
\]

We now introduce the coefficients appearing in the tree expansion.

For every $k\geq 0$ and every $n\in N(k,\mathtt{c})=\mathbb{Z}^{\nu}$, define
\begin{equation}\label{eq:base-c}
C(k,\mathtt{c})(n):=c(n),
\qquad
J(k,\mathtt{c})(t,n):=G_n(t).
\end{equation}
For every $k\geq 0$ and every $n\in N(k,\mathtt{d})=\mathbb{Z}^{\nu}$, define
\begin{equation}\label{eq:base-d}
C(k,\mathtt{d})(n):=d(n),
\qquad
J(k,\mathtt{d})(t,n):=K_n(t).
\end{equation}
For every $k\geq 1$, every $\gamma=(\gamma_1,\gamma_2)\in \big(\mathfrak{T}(k-1)\big)^2$,
and every $m=(m_1,m_2)\in N(k,\gamma)=N(k-1,\gamma_1)\times N(k-1,\gamma_2)$, define
\begin{align}
C(k,\gamma)(m)
&:= C(k-1,\gamma_1)(m_1)\,C(k-1,\gamma_2)(m_2), \label{eq:recursive-C}\\
J(k,\gamma)(t,m)
&:= -\int_0^t \Phi_{\mu(m)}(t-\tau)\,
J(k-1,\gamma_1)(\tau,m_1)\,
J(k-1,\gamma_2)(\tau,m_2)\,d\tau . \label{eq:recursive-J}
\end{align}

\begin{lemma}\label{lem:C-factorization}
For every $k\geq 0$, every $\gamma\in\mathfrak T(k)$, and every
$m\in N(k,\gamma)$, if
\[
\mathrm{Fl}_{k,\gamma}(m)=(m_1,\dots,m_{\sigma(\gamma)}),
\]
then
\begin{equation}\label{eq:C-factorization}
C(k,\gamma)(m)
=
\prod_{j=1}^{\sigma(\gamma)}
a_{\eta_{k,\gamma,j}}(m_j).
\end{equation}
\end{lemma}

\begin{proof}
We argue by induction on the tree structure.

If $\gamma=\mathtt{c}$, then $\sigma(\gamma)=1$,
\[
\mathrm{Fl}_{k,\mathtt c}(m)=(m),
\qquad
\eta_{k,\mathtt c}=(\mathtt c),
\]
and by \eqref{eq:base-c},
\[
C(k,\mathtt c)(m)=c(m)=a_{\mathtt c}(m)
=
\prod_{j=1}^{1} a_{\eta_{k,\mathtt c,j}}(m_j).
\]

If $\gamma=\mathtt{d}$, then similarly,
\[
\mathrm{Fl}_{k,\mathtt d}(m)=(m),
\qquad
\eta_{k,\mathtt d}=(\mathtt d),
\]
and by \eqref{eq:base-d},
\[
C(k,\mathtt d)(m)=d(m)=a_{\mathtt d}(m)
=
\prod_{j=1}^{1} a_{\eta_{k,\mathtt d,j}}(m_j).
\]

Now let $k\geq 1$ and $\gamma=(\gamma_1,\gamma_2)\in(\mathfrak T(k-1))^2$.
Let
\[
m=(m^{(1)},m^{(2)})\in N(k,\gamma)
=
N(k-1,\gamma_1)\times N(k-1,\gamma_2),
\]
and write
\[
\mathrm{Fl}_{k-1,\gamma_1}(m^{(1)})
=
(m_1,\dots,m_{\sigma(\gamma_1)}),
\]
\[
\mathrm{Fl}_{k-1,\gamma_2}(m^{(2)})
=
(m_{\sigma(\gamma_1)+1},\dots,m_{\sigma(\gamma)}).
\]
By \eqref{eq:recursive-C} and the induction hypothesis,
\begin{align*}
C(k,\gamma)(m)
&=
C(k-1,\gamma_1)(m^{(1)})\,C(k-1,\gamma_2)(m^{(2)}) \\
&=
\left(
\prod_{j=1}^{\sigma(\gamma_1)}
a_{\eta_{k-1,\gamma_1,j}}(m_j)
\right)
\left(
\prod_{j=1}^{\sigma(\gamma_2)}
a_{\eta_{k-1,\gamma_2,j}}(m_{\sigma(\gamma_1)+j})
\right).
\end{align*}
Since
\[
\eta_{k,\gamma}
=
\bigl(\eta_{k-1,\gamma_1},\,\eta_{k-1,\gamma_2}\bigr),
\]
the last identity is exactly \eqref{eq:C-factorization}.
\end{proof}

\subsection{Tree representation of the Picard sequence}\label{sec:tree}

In this subsection, we rewrite the Picard sequence
$\{c_k(t,n)\}_{k\geq 0}$ defined by \eqref{eq:picard-iteration-recall}
in terms of the combinatorial tree introduced in
Subsection~\ref{sec:tree-setup}.

The next lemma gives the desired tree representation of the Picard sequence.

\begin{lemma}\label{lem:tree-representation}
For every $k\geq 0$ and every $n\in \mathbb{Z}^{\nu}$, one has
\begin{equation}\label{eq:tree-representation}
c_k(t,n)
=
\sum_{\gamma\in \mathfrak{T}(k)}
\ \sum_{\substack{m\in N(k,\gamma)\\ \mu(m)=n}}
C(k,\gamma)(m)\,J(k,\gamma)(t,m).
\end{equation}
\end{lemma}

\begin{proof}
We argue formally; the validity of the rearrangements below will be justified later by
the absolute convergence estimates.

For $k=0$, since $\mathfrak{T}(0)=\mathfrak{L}=\{\mathtt{c},\mathtt{d}\}$, we have
\begin{align*}
\sum_{\gamma\in \mathfrak{T}(0)}
\ \sum_{\substack{m\in N(0,\gamma)\\ \mu(m)=n}}
C(0,\gamma)(m)\,J(0,\gamma)(t,m)
&=
C(0,\mathtt{c})(n)J(0,\mathtt{c})(t,n)
+
C(0,\mathtt{d})(n)J(0,\mathtt{d})(t,n) \\
&=
G_n(t)c(n)+K_n(t)d(n) \\
&=
c_0(t,n).
\end{align*}
Hence \eqref{eq:tree-representation} holds for $k=0$.

Now let $k\geq 1$ and assume that \eqref{eq:tree-representation} holds for $k-1$.
By \eqref{eq:picard-iteration-recall}, we have
\begin{align*}
c_k(t,n)-c_0(t,n)
&=
-\beta(n)\int_0^t K_n(t-\tau)
\sum_{\substack{n_1,n_2\in \mathbb{Z}^{\nu}\\ n_1+n_2=n}}
c_{k-1}(\tau,n_1)c_{k-1}(\tau,n_2)\,d\tau \\
&=
-\int_0^t \Phi_n(t-\tau)
\sum_{\substack{n_1,n_2\in \mathbb{Z}^{\nu}\\ n_1+n_2=n}}
c_{k-1}(\tau,n_1)c_{k-1}(\tau,n_2)\,d\tau .
\end{align*}
Substituting the induction hypothesis for each factor $c_{k-1}$, we obtain
\begin{align*}
c_k(t,n)-c_0(t,n)
&=
-\int_0^t \Phi_n(t-\tau)
\sum_{\substack{n_1,n_2\in \mathbb{Z}^{\nu}\\ n_1+n_2=n}}
\Bigg(
\sum_{\gamma_1\in \mathfrak{T}(k-1)}
\sum_{\substack{m_1\in N(k-1,\gamma_1)\\ \mu(m_1)=n_1}}
C(k-1,\gamma_1)(m_1)\,J(k-1,\gamma_1)(\tau,m_1)
\Bigg)
\\
&\qquad\qquad\qquad\qquad\times
\Bigg(
\sum_{\gamma_2\in \mathfrak{T}(k-1)}
\sum_{\substack{m_2\in N(k-1,\gamma_2)\\ \mu(m_2)=n_2}}
C(k-1,\gamma_2)(m_2)\,J(k-1,\gamma_2)(\tau,m_2)
\Bigg)\,d\tau .
\end{align*}
Rearranging the sums and using \eqref{eq:recursive-C}--\eqref{eq:recursive-J}, we get
\[
c_k(t,n)-c_0(t,n)
=
\sum_{\gamma\in (\mathfrak{T}(k-1))^2}
\ \sum_{\substack{m\in N(k,\gamma)\\ \mu(m)=n}}
C(k,\gamma)(m)\,J(k,\gamma)(t,m).
\]
Combining this with the representation of $c_0(t,n)$ coming from the two leaves
$\mathtt{c}$ and $\mathtt{d}$, we conclude that
\begin{align*}
c_k(t,n)
&=
\sum_{\gamma\in \mathfrak{L}}
\ \sum_{\substack{m\in N(k,\gamma)\\ \mu(m)=n}}
C(k,\gamma)(m)\,J(k,\gamma)(t,m)
\\
&\quad+
\sum_{\gamma\in (\mathfrak{T}(k-1))^2}
\ \sum_{\substack{m\in N(k,\gamma)\\ \mu(m)=n}}
C(k,\gamma)(m)\,J(k,\gamma)(t,m)
\\
&=
\sum_{\gamma\in \mathfrak{T}(k)}
\ \sum_{\substack{m\in N(k,\gamma)\\ \mu(m)=n}}
C(k,\gamma)(m)\,J(k,\gamma)(t,m).
\end{align*}
This proves \eqref{eq:tree-representation} for all $k\geq 0$.
\end{proof}

\begin{remark}\label{rem:tree-structure}
The combinatorial tree \eqref{eq:tree-representation} reflects precisely the structure
of the Picard iteration. The two leaves $\mathtt{c}$ and $\mathtt{d}$ encode the two
linear components generated by the initial data $(u(0,x),u_t(0,x))$, while each binary
branching corresponds to one application of the quadratic nonlinearity. The factor
$C(k,\gamma)$ records the multilinear products of the initial Fourier coefficients,
whereas $J(k,\gamma)$ collects both the propagators and the nonlinear weights.
\end{remark}

\subsection{Uniform exponential decay of the Picard sequence}\label{sec:uniform-decay}

In this subsection we prove a uniform exponential decay estimate for the Picard
sequence constructed in Subsection~\ref{sec:tree}. 


Throughout this subsection, we assume that the initial Fourier coefficients satisfy
\begin{equation}\label{eq:exp-decay-initial}
|c(n)|\leq A e^{-\rho |n|},
\qquad
|d(n)|\leq A e^{-\rho |n|},
\qquad n\in \mathbb{Z}^{\nu},
\end{equation}
where $A>0$ and $0<\rho\leq 1$. Set
\begin{equation*}
b_{\rho}:=(6\rho^{-1})^{\nu}.
\end{equation*}

Since $0\leq \Omega(n)<1$, the functions $G_n$, $K_n$, and $\Phi_n$ satisfy
\begin{equation}\label{eq:basic-propagator-bounds}
|G_n(t)|\leq 1,
\qquad
|K_n(t)|\leq t,
\qquad
|\Phi_n(t)|\leq 1,
\qquad n\in \mathbb{Z}^{\nu},\ t\geq 0.
\end{equation}

For the purpose of estimates, we now introduce three auxiliary indices. For every
$\gamma\in \bigcup_{k\geq 0}\mathfrak{T}(k)$, define recursively
\begin{equation*}
\sigma(\gamma):=
\begin{cases}
1, & \gamma=\mathtt{c}\ \text{or}\ \gamma=\mathtt{d},\\[1mm]
\sigma(\gamma_1)+\sigma(\gamma_2), & \gamma=(\gamma_1,\gamma_2),
\end{cases}
\end{equation*}
\begin{equation*}
\ell(\gamma):=
\begin{cases}
0, & \gamma=\mathtt{c},\\[1mm]
1, & \gamma=\mathtt{d},\\[1mm]
1+\ell(\gamma_1)+\ell(\gamma_2), & \gamma=(\gamma_1,\gamma_2),
\end{cases}
\end{equation*}
and
\begin{equation*}
D(\gamma):=
\begin{cases}
1, & \gamma=\mathtt{c}\ \text{or}\ \gamma=\mathtt{d},\\[1mm]
\ell(\gamma)\,D(\gamma_1)D(\gamma_2), & \gamma=(\gamma_1,\gamma_2).
\end{cases}
\end{equation*}
Here $\sigma(\gamma)$ is the number of leaves, $\ell(\gamma)$ is the effective time
degree, and $D(\gamma)$ is the denominator produced by repeated time integrations.

The next lemma gives the basic bounds for the tree coefficients.

\begin{lemma}\label{lem:CJ-estimates}
For every $k\geq 0$, every $\gamma\in \mathfrak{T}(k)$, and every $m\in N(k,\gamma)$,
one has
\begin{equation}\label{eq:C-bound}
|C(k,\gamma)(m)|\leq A^{\sigma(\gamma)} e^{-\rho |m|},
\end{equation}
and
\begin{equation}\label{eq:J-bound}
|J(k,\gamma)(t,m)|\leq \frac{t^{\ell(\gamma)}}{D(\gamma)},
\qquad t\geq 0.
\end{equation}
Here, after identifying $N(k,\gamma)$ with $(\mathbb{Z}^{\nu})^{\sigma(\gamma)}$, we write
\[
|m|:=\sum_{j=1}^{\sigma(\gamma)} |m_j|.
\]
\end{lemma}

\begin{proof}
We proceed by induction on the tree structure.

If $\gamma=\mathtt{c}$, then by \eqref{eq:base-c}, \eqref{eq:exp-decay-initial}, and
\eqref{eq:basic-propagator-bounds},
\[
|C(k,\mathtt{c})(n)|=|c(n)|\leq A e^{-\rho |n|}
= A^{\sigma(\mathtt{c})}e^{-\rho |n|},
\]
and
\[
|J(k,\mathtt{c})(t,n)|=|G_n(t)|\leq 1
=\frac{t^{\ell(\mathtt{c})}}{D(\mathtt{c})}.
\]

If $\gamma=\mathtt{d}$, then by \eqref{eq:base-d}, \eqref{eq:exp-decay-initial}, and
\eqref{eq:basic-propagator-bounds},
\[
|C(k,\mathtt{d})(n)|=|d(n)|\leq A e^{-\rho |n|}
= A^{\sigma(\mathtt{d})}e^{-\rho |n|},
\]
and
\[
|J(k,\mathtt{d})(t,n)|=|K_n(t)|\leq t
=\frac{t^{\ell(\mathtt{d})}}{D(\mathtt{d})}.
\]

Now let $k\geq 1$, let $\gamma=(\gamma_1,\gamma_2)\in (\mathfrak{T}(k-1))^2$, and let
$m=(m_1,m_2)\in N(k,\gamma)$. Assume that the estimates hold for $\gamma_1$ and
$\gamma_2$. By \eqref{eq:recursive-C} and the induction hypothesis,
\begin{align*}
|C(k,\gamma)(m)|
&=
|C(k-1,\gamma_1)(m_1)|\,|C(k-1,\gamma_2)(m_2)| \\
&\leq
A^{\sigma(\gamma_1)}e^{-\rho |m_1|}
A^{\sigma(\gamma_2)}e^{-\rho |m_2|} \\
&=
A^{\sigma(\gamma_1)+\sigma(\gamma_2)}e^{-\rho (|m_1|+|m_2|)} \\
&=
A^{\sigma(\gamma)} e^{-\rho |m|}.
\end{align*}
Similarly, by \eqref{eq:recursive-J}, \eqref{eq:basic-propagator-bounds}, and the
induction hypothesis,
\begin{align*}
|J(k,\gamma)(t,m)|
&\leq
\int_0^t
|\Phi_{\mu(m)}(t-\tau)|\,
|J(k-1,\gamma_1)(\tau,m_1)|\,
|J(k-1,\gamma_2)(\tau,m_2)|\,d\tau \\
&\leq
\int_0^t
\frac{\tau^{\ell(\gamma_1)}}{D(\gamma_1)}
\frac{\tau^{\ell(\gamma_2)}}{D(\gamma_2)}\,d\tau \\
&=
\frac{t^{1+\ell(\gamma_1)+\ell(\gamma_2)}}{
(1+\ell(\gamma_1)+\ell(\gamma_2))D(\gamma_1)D(\gamma_2)} \\
&=
\frac{t^{\ell(\gamma)}}{D(\gamma)}.
\end{align*}
This proves \eqref{eq:C-bound} and \eqref{eq:J-bound}.
\end{proof}

Next we collect several elementary combinatorial facts.

\begin{lemma}\label{lem:combinatorics}
For every $k\geq 0$ and every $\gamma\in \mathfrak{T}(k)$, the following statements hold:
\begin{enumerate}
\item
\begin{equation}\label{eq:sigma-ell}
\sigma(\gamma)\leq \ell(\gamma)+1.
\end{equation}

\item
If $0\leq \xi\leq \frac15$, then
\begin{equation}\label{eq:tree-sum}
\sum_{\gamma\in \mathfrak{T}(k)} \frac{\xi^{\ell(\gamma)}}{D(\gamma)} \leq 2,
\qquad \forall\, k\geq 0.
\end{equation}
\end{enumerate}
\end{lemma}

\begin{proof}
We first prove \eqref{eq:sigma-ell} by induction on the tree structure.
For the leaves, we have
\[
\sigma(\mathtt{c})=1\leq 0+1=\ell(\mathtt{c})+1,
\qquad
\sigma(\mathtt{d})=1\leq 1+1=\ell(\mathtt{d})+1.
\]
If $\gamma=(\gamma_1,\gamma_2)$ and the claim holds for $\gamma_1,\gamma_2$, then
\[
\sigma(\gamma)
=\sigma(\gamma_1)+\sigma(\gamma_2)
\leq (\ell(\gamma_1)+1)+(\ell(\gamma_2)+1)
=\ell(\gamma)+1.
\]
Hence \eqref{eq:sigma-ell} holds for all $\gamma$.

Finally, define
\[
\Theta_k(\xi):=
\sum_{\gamma\in \mathfrak{T}(k)} \frac{\xi^{\ell(\gamma)}}{D(\gamma)}.
\]
For $k=0$, since $\mathfrak{T}(0)=\{\mathtt{c},\mathtt{d}\}$, we have
\[
\Theta_0(\xi)
=
\frac{\xi^{\ell(\mathtt{c})}}{D(\mathtt{c})}
+
\frac{\xi^{\ell(\mathtt{d})}}{D(\mathtt{d})}
=1+\xi
\leq 2.
\]
Now let $k\geq 1$ and assume that $\Theta_{k-1}(\xi)\leq 2$. Since
\[
\mathfrak{T}(k)=\mathfrak{L}\cup \big(\mathfrak{T}(k-1)\big)^2,
\]
we obtain
\begin{align*}
\Theta_k(\xi)
&=
\sum_{\gamma\in \mathfrak{L}} \frac{\xi^{\ell(\gamma)}}{D(\gamma)}
+
\sum_{\gamma=(\gamma_1,\gamma_2)\in (\mathfrak{T}(k-1))^2}
\frac{\xi^{\ell(\gamma)}}{D(\gamma)} \\
&=
1+\xi
+
\sum_{\gamma_1,\gamma_2\in \mathfrak{T}(k-1)}
\frac{\xi^{1+\ell(\gamma_1)+\ell(\gamma_2)}}{
\ell(\gamma)\,D(\gamma_1)D(\gamma_2)} \\
&\leq
1+\xi
+
\xi
\sum_{\gamma_1,\gamma_2\in \mathfrak{T}(k-1)}
\frac{\xi^{\ell(\gamma_1)}}{D(\gamma_1)}
\frac{\xi^{\ell(\gamma_2)}}{D(\gamma_2)} \\
&=
1+\xi+\xi \Theta_{k-1}(\xi)^2 \\
&\leq
1+\xi+4\xi \\
&\leq 2,
\end{align*}
provided that $0\leq \xi\leq \frac15$. This proves \eqref{eq:tree-sum}.
\end{proof}

\begin{lemma}\label{lem:uniform-exp-decay}
Assume \eqref{eq:exp-decay-initial}. Set
\begin{equation}\label{eq:M-B-L}
M:=\max\{1,Ab_{\rho}\},
\qquad
B:=2M,
\qquad
L:=\frac{1}{5M}.
\end{equation}
Then, for every $n\in \mathbb{Z}^{\nu}$,
\begin{equation}\label{eq:uniform-exp-decay}
\sup_{0\leq t\leq L}\ \sup_{k\geq 0}\ |c_k(t,n)|
\leq
B e^{-\frac{\rho}{2}|n|}.
\end{equation}
\end{lemma}

\begin{proof}
Fix $k\geq 0$, $0\leq t\leq L$, and $n\in \mathbb{Z}^{\nu}$. By the tree representation
\eqref{eq:tree-representation} and Lemma~\ref{lem:CJ-estimates},
\begin{align*}
|c_k(t,n)|
&\leq
\sum_{\gamma\in \mathfrak{T}(k)}
\ \sum_{\substack{m\in N(k,\gamma)\\ \mu(m)=n}}
|C(k,\gamma)(m)|\,|J(k,\gamma)(t,m)| \\
&\leq
\sum_{\gamma\in \mathfrak{T}(k)}
\ \sum_{\substack{m\in N(k,\gamma)\\ \mu(m)=n}}
A^{\sigma(\gamma)} e^{-\rho |m|}
\frac{t^{\ell(\gamma)}}{D(\gamma)} .
\end{align*}
Since $\mu(m)=n$, we have $|n|=|\mu(m)|\leq |m|$, and hence
\[
e^{-\rho |m|}
=
e^{-\frac{\rho}{2}|m|}e^{-\frac{\rho}{2}|m|}
\leq
e^{-\frac{\rho}{2}|m|}e^{-\frac{\rho}{2}|n|}.
\]
Therefore,
\begin{align*}
|c_k(t,n)|
&\leq
\sum_{\gamma\in \mathfrak{T}(k)}
\frac{A^{\sigma(\gamma)} t^{\ell(\gamma)}}{D(\gamma)}
\sum_{m\in N(k,\gamma)}
e^{-\frac{\rho}{2}|m|}\,e^{-\frac{\rho}{2}|n|}.
\end{align*}
By Lemma~\ref{lem:flattening}, $N(k,\gamma)$ can be identified with
$(\mathbb{Z}^{\nu})^{\sigma(\gamma)}$, and hence
\[
\sum_{m\in N(k,\gamma)} e^{-\frac{\rho}{2}|m|}
=
\left(
\sum_{q\in \mathbb{Z}^{\nu}} e^{-\frac{\rho}{2}|q|}
\right)^{\sigma(\gamma)}.
\]
Using the one-dimensional estimate
\[
\sum_{m\in \mathbb{Z}} e^{-\frac{\rho}{2}|m|}\leq 6\rho^{-1},
\]
we obtain
\[
\sum_{q\in \mathbb{Z}^{\nu}} e^{-\frac{\rho}{2}|q|}
=
\left(\sum_{m\in \mathbb{Z}} e^{-\frac{\rho}{2}|m|}\right)^{\nu}
\leq
(6\rho^{-1})^{\nu}
=
b_{\rho}.
\]
Hence
\begin{align*}
|c_k(t,n)|
&\leq
\sum_{\gamma\in \mathfrak{T}(k)}
\frac{(Ab_{\rho})^{\sigma(\gamma)} t^{\ell(\gamma)}}{D(\gamma)}
e^{-\frac{\rho}{2}|n|}.
\end{align*}
Now, since $Ab_{\rho}\leq M$ and $\sigma(\gamma)\leq \ell(\gamma)+1$, we find
\[
(Ab_{\rho})^{\sigma(\gamma)} t^{\ell(\gamma)}
\leq
M^{\sigma(\gamma)} t^{\ell(\gamma)}
\leq
M^{\ell(\gamma)+1} t^{\ell(\gamma)}
=
M(Mt)^{\ell(\gamma)}.
\]
Therefore,
\[
|c_k(t,n)|
\leq
M
\sum_{\gamma\in \mathfrak{T}(k)}
\frac{(Mt)^{\ell(\gamma)}}{D(\gamma)}
e^{-\frac{\rho}{2}|n|}.
\]
Since $0\leq t\leq L=\frac{1}{5M}$, we have $0\leq Mt\leq \frac15$. By
\eqref{eq:tree-sum}, it follows that
\[
|c_k(t,n)|
\leq
2M e^{-\frac{\rho}{2}|n|}
=
B e^{-\frac{\rho}{2}|n|}.
\]
This proves \eqref{eq:uniform-exp-decay}.
\end{proof}

\begin{remark}\label{rem:constants}
The constants in \eqref{eq:M-B-L} are not optimized. Their role is merely to provide
a convenient uniform estimate without separating the cases $Ab_{\rho}\geq 1$ and
$Ab_{\rho}<1$. 
\end{remark}

\subsection{Convergence of the Picard sequence}\label{sec:cauchy}

We now prove that the Picard sequence $\{c_k(t,n)\}_{k\geq 0}$ converges on
$[0,L]\times \mathbb{Z}^{\nu}$, where $L$ is the time scale introduced in
Lemma~\ref{lem:uniform-exp-decay}. The key point is to establish an explicit bound
for the difference $c_k(t,n)-c_{k-1}(t,n)$.

For every $m\geq 2$ and every $n\in \mathbb{Z}^{\nu}$, define
\begin{equation}\label{eq:Em-definition}
\mathcal{E}_m(n)
:=
\sum_{\substack{n_1,\dots,n_m\in \mathbb{Z}^{\nu}\\ n_1+\cdots+n_m=n}}
\prod_{j=1}^{m} e^{-\frac{\rho}{2}|n_j|}.
\end{equation}

The following lemma is the analogue of the convergence estimate in the first-order case.

\begin{lemma}\label{lem:cauchy-estimate}
Let $B$ and $L$ be as in \eqref{eq:M-B-L}. Then, for every $k\geq 1$, every
$0\leq t\leq L$, and every $n\in \mathbb{Z}^{\nu}$, one has
\begin{equation}\label{eq:cauchy-estimate-1}
|c_k(t,n)-c_{k-1}(t,n)|
\leq
\frac{2^{k-1}B^{k+1}t^k}{k!}\,\mathcal{E}_{k+1}(n).
\end{equation}
Moreover, setting
\begin{equation*}
\widetilde b_{\rho}:=(12\rho^{-1})^{\nu},
\end{equation*}
we also have
\begin{equation}\label{eq:cauchy-estimate-2}
|c_k(t,n)-c_{k-1}(t,n)|
\leq
\frac{B\widetilde b_{\rho}}{2}\,
\frac{(2B\widetilde b_{\rho}t)^k}{k!}\,
e^{-\frac{\rho}{4}|n|}.
\end{equation}
In particular, $\{c_k(t,n)\}_{k\geq 0}$ is a Cauchy sequence on
$[0,L]\times \mathbb{Z}^{\nu}$.
\end{lemma}

\begin{proof}
We proceed by induction on $k$.

For $k=1$, by \eqref{eq:picard-iteration-recall}, \eqref{eq:Phi-definition}, and
Lemma~\ref{lem:uniform-exp-decay}, we have
\begin{align*}
|c_1(t,n)-c_0(t,n)|
&=
\left|
\int_0^t \Phi_n(t-\tau)
\sum_{\substack{n_1,n_2\in \mathbb{Z}^{\nu}\\ n_1+n_2=n}}
c_0(\tau,n_1)c_0(\tau,n_2)\,d\tau
\right| \\
&\leq
\int_0^t
\sum_{\substack{n_1,n_2\in \mathbb{Z}^{\nu}\\ n_1+n_2=n}}
|c_0(\tau,n_1)|\,|c_0(\tau,n_2)|\,d\tau \\
&\leq
B^2 t
\sum_{\substack{n_1,n_2\in \mathbb{Z}^{\nu}\\ n_1+n_2=n}}
e^{-\frac{\rho}{2}|n_1|}e^{-\frac{\rho}{2}|n_2|}.
\end{align*}
This is exactly \eqref{eq:cauchy-estimate-1} for $k=1$.

Now let $k\geq 2$ and assume that \eqref{eq:cauchy-estimate-1} holds for $k-1$.
Using \eqref{eq:picard-iteration-recall}, we obtain
\begin{align*}
&|c_k(t,n)-c_{k-1}(t,n)| \\
&\qquad=
\left|
\int_0^t \Phi_n(t-\tau)
\sum_{\substack{n_1,n_2\in \mathbb{Z}^{\nu}\\ n_1+n_2=n}}
\Big(
c_{k-1}(\tau,n_1)c_{k-1}(\tau,n_2)
-
c_{k-2}(\tau,n_1)c_{k-2}(\tau,n_2)
\Big)\,d\tau
\right| \\
&\qquad\leq
\int_0^t
\sum_{\substack{n_1,n_2\in \mathbb{Z}^{\nu}\\ n_1+n_2=n}}
\big|
c_{k-1}(\tau,n_1)c_{k-1}(\tau,n_2)
-
c_{k-2}(\tau,n_1)c_{k-2}(\tau,n_2)
\big|\,d\tau \\
&\qquad\leq
\int_0^t
\sum_{\substack{n_1,n_2\in \mathbb{Z}^{\nu}\\ n_1+n_2=n}}
|c_{k-1}(\tau,n_1)|\,|c_{k-1}(\tau,n_2)-c_{k-2}(\tau,n_2)|\,d\tau
\\
&\qquad\quad+
\int_0^t
\sum_{\substack{n_1,n_2\in \mathbb{Z}^{\nu}\\ n_1+n_2=n}}
|c_{k-1}(\tau,n_1)-c_{k-2}(\tau,n_1)|\,|c_{k-2}(\tau,n_2)|\,d\tau \\
&\qquad=: (I)+(II).
\end{align*}
Since $0\leq \tau\leq t\leq L$, Lemma~\ref{lem:uniform-exp-decay} applies at time
$\tau$.

For the first term, by Lemma~\ref{lem:uniform-exp-decay} and the induction hypothesis,
\begin{align*}
(I)
&\leq
\int_0^t
\sum_{\substack{n_1,n_2\in \mathbb{Z}^{\nu}\\ n_1+n_2=n}}
B e^{-\frac{\rho}{2}|n_1|}
\cdot
\frac{2^{k-2}B^k\tau^{k-1}}{(k-1)!}
\sum_{\substack{m_1,\dots,m_k\in \mathbb{Z}^{\nu}\\ m_1+\cdots+m_k=n_2}}
\prod_{j=1}^{k} e^{-\frac{\rho}{2}|m_j|}
\,d\tau \\
&=
\frac{2^{k-2}B^{k+1}}{(k-1)!}
\int_0^t \tau^{k-1}\,d\tau
\sum_{\substack{n_1,n_2\in \mathbb{Z}^{\nu}\\ n_1+n_2=n}}
\sum_{\substack{m_1,\dots,m_k\in \mathbb{Z}^{\nu}\\ m_1+\cdots+m_k=n_2}}
e^{-\frac{\rho}{2}|n_1|}
\prod_{j=1}^{k} e^{-\frac{\rho}{2}|m_j|} \\
&=
\frac{2^{k-2}B^{k+1}t^k}{k!}\,\mathcal{E}_{k+1}(n).
\end{align*}
By the same argument,
\[
(II)\leq \frac{2^{k-2}B^{k+1}t^k}{k!}\,\mathcal{E}_{k+1}(n).
\]
Hence
\[
|c_k(t,n)-c_{k-1}(t,n)|
\leq
(I)+(II)
\leq
\frac{2^{k-1}B^{k+1}t^k}{k!}\,\mathcal{E}_{k+1}(n),
\]
which proves \eqref{eq:cauchy-estimate-1} for all $k\geq 1$.

We next prove \eqref{eq:cauchy-estimate-2}. By the triangle inequality,
\[
|n|
=
\left| \sum_{j=1}^{k+1} n_j \right|
\leq
\sum_{j=1}^{k+1} |n_j|,
\]
and therefore
\[
\prod_{j=1}^{k+1} e^{-\frac{\rho}{2}|n_j|}
=
\prod_{j=1}^{k+1} e^{-\frac{\rho}{4}|n_j|}
\cdot
\prod_{j=1}^{k+1} e^{-\frac{\rho}{4}|n_j|}
\leq
\left(
\prod_{j=1}^{k+1} e^{-\frac{\rho}{4}|n_j|}
\right)
e^{-\frac{\rho}{4}|n|}.
\]
Using this in \eqref{eq:Em-definition}, we find
\begin{align*}
\mathcal{E}_{k+1}(n)
&\leq
\sum_{n_1,\dots,n_{k+1}\in \mathbb{Z}^{\nu}}
\prod_{j=1}^{k+1} e^{-\frac{\rho}{4}|n_j|}
\cdot
e^{-\frac{\rho}{4}|n|} \\
&=
\left(
\sum_{q\in \mathbb{Z}^{\nu}} e^{-\frac{\rho}{4}|q|}
\right)^{k+1}
e^{-\frac{\rho}{4}|n|}.
\end{align*}
By the one-dimensional estimate used in the previous subsection,
\[
\sum_{m\in \mathbb{Z}} e^{-\frac{\rho}{4}|m|}\leq 12\rho^{-1},
\]
and hence
\[
\sum_{q\in \mathbb{Z}^{\nu}} e^{-\frac{\rho}{4}|q|}
=
\left(
\sum_{m\in \mathbb{Z}} e^{-\frac{\rho}{4}|m|}
\right)^\nu
\leq
(12\rho^{-1})^\nu
=
\widetilde b_{\rho}.
\]
Thus,
\[
\mathcal{E}_{k+1}(n)
\leq
\widetilde b_{\rho}^{\,k+1} e^{-\frac{\rho}{4}|n|}.
\]
Substituting this into \eqref{eq:cauchy-estimate-1}, we obtain
\begin{align*}
|c_k(t,n)-c_{k-1}(t,n)|
&\leq
\frac{2^{k-1}B^{k+1}t^k}{k!}\,
\widetilde b_{\rho}^{\,k+1}
e^{-\frac{\rho}{4}|n|} \\
&=
\frac{B\widetilde b_{\rho}}{2}\,
\frac{(2B\widetilde b_{\rho}t)^k}{k!}\,
e^{-\frac{\rho}{4}|n|},
\end{align*}
which is exactly \eqref{eq:cauchy-estimate-2}.

Finally, for $k\geq 1$, define
\[
M_k:=\frac{B\widetilde b_{\rho}}{2}\,
\frac{(2B\widetilde b_{\rho}L)^k}{k!}.
\]
Then, by \eqref{eq:cauchy-estimate-2}, for every $0\leq t\leq L$ and every
$n\in \mathbb{Z}^{\nu}$,
\[
|c_k(t,n)-c_{k-1}(t,n)|\leq M_k.
\]
It follows from
\[
\sum_{k=1}^{\infty} M_k
=
\frac{B\widetilde b_{\rho}}{2}
\sum_{k=1}^{\infty}
\frac{(2B\widetilde b_{\rho}L)^k}{k!}
<\infty,
\]
that the series
\[
\sum_{k=1}^{\infty} \big(c_k(t,n)-c_{k-1}(t,n)\big)
\]
converges uniformly on $[0,L]\times \mathbb{Z}^{\nu}$. Equivalently,
$\{c_k(t,n)\}_{k\geq 0}$ is a Cauchy sequence on $[0,L]\times \mathbb{Z}^{\nu}$.
This completes the proof.
\end{proof}

\begin{remark}\label{rem:cauchy-remark}
The estimate \eqref{eq:cauchy-estimate-2} is not optimized. Its purpose is only to
obtain factorial decay in $k$, which is sufficient for the convergence of the Picard
sequence. 

\end{remark}

\begin{proof}[Proof of Theorem \ref{thm:exponential decay}]
We divide the proof into three steps.

\medskip
\noindent
\textbf{Step 1: Construction of the limit Fourier coefficients.}
By Lemma~\ref{lem:cauchy-estimate}, the Picard sequence
$\{c_k(t,n)\}_{k\geq 0}$ is a Cauchy sequence on $[0,L]\times \mathbb{Z}^{\nu}$.
Hence there exists a limit function, denoted by $c^{\dagger}(t,n)$, such that
\[
c_k(t,n)\longrightarrow c^{\dagger}(t,n)
\qquad \text{uniformly on } [0,L]\times \mathbb{Z}^{\nu}.
\]
Moreover, by Lemma~\ref{lem:uniform-exp-decay},
\[
|c_k(t,n)|\leq B e^{-\frac{\rho}{2}|n|}
\qquad
\text{for all } k\geq 0,\ 0\leq t\leq L,\ n\in \mathbb{Z}^{\nu}.
\]
Passing to the limit yields
\begin{equation}\label{eq:limit-decay}
|c^{\dagger}(t,n)|\leq B e^{-\frac{\rho}{2}|n|},
\qquad
0\leq t\leq L,\ \ n\in \mathbb{Z}^{\nu}.
\end{equation}

We claim that $c^{\dagger}$ satisfies the integral equation
\begin{equation}\label{eq:limit-integral-equation}
c^{\dagger}(t,n)
=
c_0(t,n)
-
\int_0^t \Phi_n(t-\tau)
\sum_{\substack{n_1,n_2\in \mathbb{Z}^{\nu}\\ n_1+n_2=n}}
c^{\dagger}(\tau,n_1)c^{\dagger}(\tau,n_2)\,d\tau,
\end{equation}
where
\[
c_0(t,n)=G_n(t)c(n)+K_n(t)d(n).
\]
Indeed, for every fixed $(t,n)\in [0,L]\times \mathbb{Z}^{\nu}$,
\[
c_k(t,n)
=
c_0(t,n)
-
\int_0^t \Phi_n(t-\tau)
\sum_{\substack{n_1,n_2\in \mathbb{Z}^{\nu}\\ n_1+n_2=n}}
c_{k-1}(\tau,n_1)c_{k-1}(\tau,n_2)\,d\tau.
\]
Since $c_{k-1}\to c^{\dagger}$ uniformly and
\[
|c_{k-1}(\tau,n_1)c_{k-1}(\tau,n_2)|
\leq
B^2 e^{-\frac{\rho}{2}|n_1|}e^{-\frac{\rho}{2}|n_2|},
\]
the series over $n_1+n_2=n$ is absolutely convergent and dominated by a summable
majorant independent of $k$ and $\tau$. Therefore, by dominated convergence, we may
pass to the limit in the Picard iteration and obtain
\eqref{eq:limit-integral-equation}.

\medskip
\noindent
\textbf{Step 2: Existence of a classical solution.}
Define
\begin{equation*}
u^{\dagger}(t,x)
:=
\sum_{n\in \mathbb{Z}^{\nu}} c^{\dagger}(t,n)e^{i\langle n,\omega\rangle x}.
\end{equation*}
By \eqref{eq:limit-decay}, the series for $u^{\dagger}$ converges absolutely and
uniformly on $[0,L]\times \mathbb{R}$.

For every fixed $n\in \mathbb{Z}^{\nu}$, the function $t\mapsto c_k(t,n)$ is continuous
for every $k\geq 0$ by construction of the Picard sequence. Since
$c_k(\cdot,n)\to c^{\dagger}(\cdot,n)$ uniformly on $[0,L]$, it follows that
\begin{equation}\label{eq:c-dagger-continuous}
c^{\dagger}(\cdot,n)\in C([0,L]),
\qquad \forall\, n\in \mathbb{Z}^{\nu}.
\end{equation}

Next define, for $0\leq t\leq L$ and $n\in \mathbb{Z}^{\nu}$,
\begin{equation}\label{eq:nonlinear-coefficients}
\mathcal{N}^{\dagger}(t,n)
:=
\sum_{\substack{n_1,n_2\in \mathbb{Z}^{\nu}\\ n_1+n_2=n}}
c^{\dagger}(t,n_1)c^{\dagger}(t,n_2).
\end{equation}
By \eqref{eq:limit-decay},
\[
|\mathcal{N}^{\dagger}(t,n)|
\leq
B^2
\sum_{\substack{n_1,n_2\in \mathbb{Z}^{\nu}\\ n_1+n_2=n}}
e^{-\frac{\rho}{2}|n_1|}e^{-\frac{\rho}{2}|n_2|}.
\]
Arguing as in Lemma~\ref{lem:cauchy-estimate}, we obtain
\begin{equation}\label{eq:N-dagger-decay}
|\mathcal{N}^{\dagger}(t,n)|
\leq
B^2 \widetilde b_{\rho}^{\,2} e^{-\frac{\rho}{4}|n|},
\qquad
0\leq t\leq L,\ \ n\in \mathbb{Z}^{\nu},
\end{equation}
where $\widetilde b_{\rho}=(12\rho^{-1})^{\nu}$.

For each fixed $n$, the series in \eqref{eq:nonlinear-coefficients} is uniformly
absolutely convergent on $[0,L]$ by \eqref{eq:limit-decay}. Since each summand
$t\mapsto c^{\dagger}(t,n_1)c^{\dagger}(t,n_2)$ is continuous by
\eqref{eq:c-dagger-continuous}, it follows from the Weierstrass M-test that
\begin{equation*}
\mathcal{N}^{\dagger}(\cdot,n)\in C([0,L]),
\qquad \forall\, n\in \mathbb{Z}^{\nu}.
\end{equation*}

Now \eqref{eq:limit-integral-equation} is precisely the Duhamel formula for the linear
equation
\begin{equation}\label{eq:coefficient-ode-final}
\partial_t^2 c^{\dagger}(t,n)
+\Omega(n)^2 c^{\dagger}(t,n)
=
-\Omega(n)^2 \mathcal{N}^{\dagger}(t,n),
\end{equation}
with initial conditions
\begin{equation}\label{eq:coefficient-initial-final}
c^{\dagger}(0,n)=c(n),
\qquad
\partial_t c^{\dagger}(0,n)=d(n).
\end{equation}
Since $\mathcal{N}^{\dagger}(\cdot,n)\in C([0,L])$, standard ODE theory for the forced
harmonic oscillator yields
\begin{equation*}
c^{\dagger}(\cdot,n)\in C^2([0,L]),
\qquad \forall\, n\in \mathbb{Z}^{\nu},
\end{equation*}
and \eqref{eq:coefficient-ode-final}--\eqref{eq:coefficient-initial-final} hold pointwise.

We next derive bounds for the time derivatives. From
\eqref{eq:coefficient-ode-final} and \eqref{eq:N-dagger-decay}, using
$0\leq \Omega(n)^2\leq 1$, we obtain
\begin{equation}\label{eq:ctt-bound}
|\partial_t^2 c^{\dagger}(t,n)|
\leq
|c^{\dagger}(t,n)|+|\mathcal{N}^{\dagger}(t,n)|
\leq
B e^{-\frac{\rho}{2}|n|}
+
B^2 \widetilde b_{\rho}^{\,2} e^{-\frac{\rho}{4}|n|}
\leq
C_1 e^{-\frac{\rho}{4}|n|}
\end{equation}
for some constant $C_1>0$.

Differentiating \eqref{eq:limit-integral-equation} once and using
$\partial_t \Phi_n(t)=\beta(n)\partial_t K_n(t)$, we find
\begin{align*}
\partial_t c^{\dagger}(t,n)
&=
\partial_t G_n(t)\,c(n)
+
\partial_t K_n(t)\,d(n)
-
\int_0^t \partial_t \Phi_n(t-\tau)\,\mathcal{N}^{\dagger}(\tau,n)\,d\tau .
\end{align*}
Since $|\partial_t G_n(t)|\leq 1$, $|\partial_t K_n(t)|\leq 1$, and
$|\partial_t \Phi_n(t)|\leq 1$, it follows from
\eqref{eq:N-dagger-decay} that
\begin{equation}\label{eq:ct-bound}
|\partial_t c^{\dagger}(t,n)|
\leq
A e^{-\rho|n|}
+
A e^{-\rho|n|}
+
L B^2 \widetilde b_{\rho}^{\,2} e^{-\frac{\rho}{4}|n|}
\leq
C_2 e^{-\frac{\rho}{4}|n|}
\end{equation}
for some constant $C_2>0$.

Let
\[
C_{\omega}:=\max_{1\leq j\leq \nu} |\omega_j|.
\]
Since $|n|=\sum_{j=1}^{\nu}|n_j|$, we have
\begin{equation*}
|\langle n,\omega\rangle|
=
\left|\sum_{j=1}^{\nu} n_j\omega_j\right|
\leq
\sum_{j=1}^{\nu} |n_j||\omega_j|
\leq
C_{\omega}|n|.
\end{equation*}
Therefore,
\[
1+|\langle n,\omega\rangle|^2 \leq (1+C_{\omega}^2)(1+|n|^2).
\]
Combining this with \eqref{eq:limit-decay}, \eqref{eq:N-dagger-decay},
\eqref{eq:ct-bound}, and \eqref{eq:ctt-bound}, we get
\[
\sum_{n\in \mathbb{Z}^{\nu}}
\big(1+|\langle n,\omega\rangle|^2\big)
\Big(
|c^{\dagger}(t,n)|+|\partial_t c^{\dagger}(t,n)|
+|\partial_t^2 c^{\dagger}(t,n)|+|\mathcal{N}^{\dagger}(t,n)|
\Big)
<\infty
\]
uniformly for $t\in [0,L]$. Hence the Fourier series defining
$u^{\dagger}$, $u_t^{\dagger}$, $u_{tt}^{\dagger}$, $u_{xx}^{\dagger}$,
$u_{xxtt}^{\dagger}$, and $(u^{\dagger})^2_{xx}$ all converge absolutely and uniformly
on $[0,L]\times \mathbb{R}$. Therefore termwise differentiation is justified, and
\eqref{eq:coefficient-ode-final} implies that $u^{\dagger}$ is a classical solution of
\eqref{eq:IBq-intro}. The initial conditions follow from \eqref{eq:coefficient-initial-final}.
Finally, by \eqref{eq:limit-decay}, the solution $u^{\dagger}$ satisfies
\eqref{eq:IBq-intro} with $\widehat u(t,n)=c^{\dagger}(t,n)$.

\medskip
\noindent
\textbf{Step 3: Uniqueness.}
Let
\[
v(t,x)=\sum_{n\in \mathbb{Z}^{\nu}} \widehat v(t,n)e^{i\langle n,\omega\rangle x},
\qquad
w(t,x)=\sum_{n\in \mathbb{Z}^{\nu}} \widehat w(t,n)e^{i\langle n,\omega\rangle x},
\]
be two spatially quasi-periodic classical solutions satisfying
\begin{equation}\label{eq:uniqueness-decay}
|\widehat v(t,n)|\leq B e^{-\frac{\rho}{2}|n|},
\qquad
|\widehat w(t,n)|\leq B e^{-\frac{\rho}{2}|n|},
\qquad
0\leq t\leq L,\ \ n\in \mathbb{Z}^{\nu}.
\end{equation}
By the same reduction argument as in Section~\ref{sec:reduction}, which is justified
here by the exponential decay \eqref{eq:uniqueness-decay}, both $\widehat v$ and
$\widehat w$ satisfy the coefficient equation
\[
\partial_t^2 a(t,n)+\Omega(n)^2 a(t,n)
=
-\Omega(n)^2
\sum_{\substack{n_1,n_2\in \mathbb{Z}^{\nu}\\ n_1+n_2=n}}
a(t,n_1)a(t,n_2),
\]
together with the same initial conditions. Equivalently, both satisfy the same integral
equation
\[
a(t,n)
=
G_n(t)a(0,n)+K_n(t)\partial_t a(0,n)
-
\int_0^t \Phi_n(t-\tau)
\sum_{\substack{n_1,n_2\in \mathbb{Z}^{\nu}\\ n_1+n_2=n}}
a(\tau,n_1)a(\tau,n_2)\,d\tau .
\]
Since $v$ and $w$ have the same initial data, we obtain
\begin{align}
\widehat v(t,n)-\widehat w(t,n)
&=
-\int_0^t \Phi_n(t-\tau)
\sum_{\substack{n_1,n_2\in \mathbb{Z}^{\nu}\\ n_1+n_2=n}}
\Big(
\widehat v(\tau,n_1)\widehat v(\tau,n_2)
-
\widehat w(\tau,n_1)\widehat w(\tau,n_2)
\Big)\,d\tau .
\label{eq:difference-integral}
\end{align}

Let
\[
\delta(t,n):=\widehat v(t,n)-\widehat w(t,n),
\qquad
0\leq t\leq L,\ \ n\in \mathbb{Z}^{\nu}.
\]
Then \eqref{eq:difference-integral} becomes
\begin{equation}\label{eq:difference-integral-delta}
\delta(t,n)
=
-\int_0^t \Phi_n(t-\tau)
\sum_{\substack{n_1,n_2\in \mathbb{Z}^{\nu}\\ n_1+n_2=n}}
\Big(
\delta(\tau,n_1)\widehat w(\tau,n_2)
+
\widehat v(\tau,n_1)\delta(\tau,n_2)
\Big)\,d\tau .
\end{equation}
Since \eqref{eq:uniqueness-decay} holds, we also have the crude bound
\begin{equation}\label{eq:delta-crude}
|\delta(t,n)|
\leq
|\widehat v(t,n)|+|\widehat w(t,n)|
\leq
2B e^{-\frac{\rho}{2}|n|},
\qquad
0\leq t\leq L,\ \ n\in \mathbb{Z}^{\nu}.
\end{equation}

We claim that, for every $k\geq 1$,
\begin{equation}\label{eq:uniqueness-claim}
|\delta(t,n)|
\leq
\frac{2^{k+1}B^{k+1}t^k}{k!}\,\mathcal{E}_{k+1}(n),
\qquad
0\leq t\leq L,\ \ n\in \mathbb{Z}^{\nu},
\end{equation}
where
\[
\mathcal{E}_{m}(n)
:=
\sum_{\substack{n_1,\dots,n_m\in \mathbb{Z}^{\nu}\\ n_1+\cdots+n_m=n}}
\prod_{j=1}^{m} e^{-\frac{\rho}{2}|n_j|}.
\]

For $k=1$, by \eqref{eq:difference-integral-delta}, \eqref{eq:uniqueness-decay},
and \eqref{eq:delta-crude}, we obtain
\begin{align*}
|\delta(t,n)|
&\leq
\int_0^t |\Phi_n(t-\tau)|
\sum_{\substack{n_1,n_2\in \mathbb{Z}^{\nu}\\ n_1+n_2=n}}
\Big(
|\delta(\tau,n_1)|\,|\widehat w(\tau,n_2)|
+
|\widehat v(\tau,n_1)|\,|\delta(\tau,n_2)|
\Big)\,d\tau \\
&\leq
\int_0^t
\sum_{\substack{n_1,n_2\in \mathbb{Z}^{\nu}\\ n_1+n_2=n}}
\Big(
2B e^{-\frac{\rho}{2}|n_1|}\cdot B e^{-\frac{\rho}{2}|n_2|}
+
B e^{-\frac{\rho}{2}|n_1|}\cdot 2B e^{-\frac{\rho}{2}|n_2|}
\Big)\,d\tau \\
&=
4B^2 t
\sum_{\substack{n_1,n_2\in \mathbb{Z}^{\nu}\\ n_1+n_2=n}}
e^{-\frac{\rho}{2}|n_1|}e^{-\frac{\rho}{2}|n_2|} \\
&=
\frac{2^{2}B^{2}t}{1!}\,\mathcal{E}_{2}(n).
\end{align*}
Hence \eqref{eq:uniqueness-claim} holds for $k=1$.

Now let $k\geq 2$ and assume that \eqref{eq:uniqueness-claim} holds with $k-1$
in place of $k$, that is,
\[
|\delta(\tau,m)|
\leq
\frac{2^{k}B^{k}\tau^{k-1}}{(k-1)!}\,\mathcal{E}_{k}(m),
\qquad
0\leq \tau\leq L,\ \ m\in \mathbb{Z}^{\nu}.
\]
Using \eqref{eq:difference-integral-delta}, \eqref{eq:uniqueness-decay}, and the
induction hypothesis, we get
\begin{align*}
|\delta(t,n)|
&\leq
\int_0^t
\sum_{\substack{n_1,n_2\in \mathbb{Z}^{\nu}\\ n_1+n_2=n}}
|\delta(\tau,n_1)|\,|\widehat w(\tau,n_2)|\,d\tau \\
&\quad+
\int_0^t
\sum_{\substack{n_1,n_2\in \mathbb{Z}^{\nu}\\ n_1+n_2=n}}
|\widehat v(\tau,n_1)|\,|\delta(\tau,n_2)|\,d\tau \\
&\leq
\int_0^t
\sum_{\substack{n_1,n_2\in \mathbb{Z}^{\nu}\\ n_1+n_2=n}}
\frac{2^{k}B^{k}\tau^{k-1}}{(k-1)!}\,\mathcal{E}_{k}(n_1)\,
B e^{-\frac{\rho}{2}|n_2|}\,d\tau \\
&\quad+
\int_0^t
\sum_{\substack{n_1,n_2\in \mathbb{Z}^{\nu}\\ n_1+n_2=n}}
B e^{-\frac{\rho}{2}|n_1|}\,
\frac{2^{k}B^{k}\tau^{k-1}}{(k-1)!}\,\mathcal{E}_{k}(n_2)\,d\tau .
\end{align*}
By the definition of $\mathcal{E}_{k}$ and a relabeling of the frequency variables,
\[
\sum_{\substack{n_1,n_2\in \mathbb{Z}^{\nu}\\ n_1+n_2=n}}
\mathcal{E}_{k}(n_1)e^{-\frac{\rho}{2}|n_2|}
=
\mathcal{E}_{k+1}(n),
\]
and similarly
\[
\sum_{\substack{n_1,n_2\in \mathbb{Z}^{\nu}\\ n_1+n_2=n}}
e^{-\frac{\rho}{2}|n_1|}\mathcal{E}_{k}(n_2)
=
\mathcal{E}_{k+1}(n).
\]
Therefore
\begin{align*}
|\delta(t,n)|
&\leq
\frac{2^{k}B^{k+1}}{(k-1)!}
\int_0^t \tau^{k-1}\,d\tau \, \mathcal{E}_{k+1}(n)
+
\frac{2^{k}B^{k+1}}{(k-1)!}
\int_0^t \tau^{k-1}\,d\tau \, \mathcal{E}_{k+1}(n) \\
&=
\frac{2^{k}B^{k+1}t^k}{k!}\,\mathcal{E}_{k+1}(n)
+
\frac{2^{k}B^{k+1}t^k}{k!}\,\mathcal{E}_{k+1}(n) \\
&=
\frac{2^{k+1}B^{k+1}t^k}{k!}\,\mathcal{E}_{k+1}(n).
\end{align*}
Thus \eqref{eq:uniqueness-claim} holds for all $k\geq 1$.

By the same argument as in Lemma~\ref{lem:cauchy-estimate},
\[
\mathcal{E}_{k+1}(n)
\leq
\widetilde b_{\rho}^{\,k+1} e^{-\frac{\rho}{4}|n|}.
\]
Hence \eqref{eq:uniqueness-claim} implies
\[
|\delta(t,n)|
\leq
2B\widetilde b_{\rho}\,
\frac{(2B\widetilde b_{\rho}t)^k}{k!}
e^{-\frac{\rho}{4}|n|},
\qquad
k\geq 1.
\]
For every fixed $(t,n)\in [0,L]\times \mathbb{Z}^{\nu}$, the right-hand side tends to
zero as $k\to\infty$. Therefore
\[
\delta(t,n)\equiv 0,
\qquad
0\leq t\leq L,\ \ n\in \mathbb{Z}^{\nu}.
\]
That is,
\[
\widehat v(t,n)\equiv \widehat w(t,n),
\qquad
0\leq t\leq L,\ \ n\in \mathbb{Z}^{\nu}.
\]
Hence $v\equiv w$, proving uniqueness.

Combining the three steps, the proof of Theorem~\ref{thm:exponential decay} is complete.
\end{proof}

\section{Polynomial decay and proof of Theorem \ref{thm:polynomial decay}}\label{sec:3}

In this section we assume that the initial Fourier coefficients satisfy
\begin{equation}\label{eq:poly-initial-cd}
|c(n)|\le A(1+|n|)^{-r},
\qquad
|d(n)|\le A(1+|n|)^{-r},
\qquad n\in\mathbb Z^\nu,
\end{equation}
for some constants $A>0$ and $r>0$. We retain the notation introduced in
Subsections~\ref{sec:tree-setup}--\ref{sec:cauchy}. In particular, the flattening map
$\mathrm{Fl}_{k,\gamma}$, the leaf-label sequence $\eta_{k,\gamma}$, the factorization
formula \eqref{eq:C-factorization}, the estimate \eqref{eq:J-bound}, the inequality
\eqref{eq:sigma-ell}, and the tree-sum bound \eqref{eq:tree-sum} remain valid.

Thus, in the polynomial-decay setting, the only genuinely new ingredient is a different
estimate for the tree coefficient $C(k,\gamma)$.

For every $n\in\mathbb Z^\nu$, define
\[
w_r(n):=(1+|n|)^{-r}.
\]
We also introduce the weighted supremum norm
\[
\|f\|_{X_r}:=\sup_{n\in\mathbb Z^\nu}(1+|n|)^r|f(n)|.
\]

\subsection{Weighted convolution and polynomial tree bounds}

The main new ingredient in the polynomial setting is the following weighted convolution
estimate.

\begin{lemma}\label{lem:weighted-convolution}
Assume that $r>\nu$. Then
\begin{equation}\label{eq:weighted-convolution}
\sum_{\substack{n_1,n_2\in\mathbb Z^\nu\\ n_1+n_2=n}}
w_r(n_1)w_r(n_2)
\le
K_{r,\nu}\,w_r(n),
\qquad n\in\mathbb Z^\nu,
\end{equation}
where
\[
K_{r,\nu}:=2^{r+1}H(r;\nu),
\qquad
H(r;\nu):=\sum_{m\in\mathbb Z^\nu}(1+|m|)^{-r}.
\]
Consequently, for any $f,g\in X_r$,
\begin{equation}\label{eq:Xr-algebra}
\sup_{n\in\mathbb Z^\nu}(1+|n|)^r
\sum_{\substack{n_1,n_2\in\mathbb Z^\nu\\ n_1+n_2=n}}
|f(n_1)||g(n_2)|
\le
K_{r,\nu}\,\|f\|_{X_r}\|g\|_{X_r}.
\end{equation}
\end{lemma}

\begin{proof}
Fix $n\in\mathbb Z^\nu$. If $n_1+n_2=n$, then at least one of the inequalities
\[
|n_1|\ge \frac{|n|}{2},
\qquad
|n_2|\ge \frac{|n|}{2}
\]
must hold. Therefore,
\[
w_r(n_1)w_r(n_2)
\le
2^r w_r(n)\bigl(w_r(n_1)+w_r(n_2)\bigr).
\]
Summing over all pairs $(n_1,n_2)$ with $n_1+n_2=n$, we obtain
\begin{align*}
\sum_{\substack{n_1,n_2\in\mathbb Z^\nu\\ n_1+n_2=n}}
w_r(n_1)w_r(n_2)
&\le
2^r w_r(n)
\sum_{\substack{n_1,n_2\in\mathbb Z^\nu\\ n_1+n_2=n}}
\bigl(w_r(n_1)+w_r(n_2)\bigr) \\
&=
2^{r+1}w_r(n)\sum_{m\in\mathbb Z^\nu}w_r(m) \\
&=
K_{r,\nu}w_r(n).
\end{align*}
This proves \eqref{eq:weighted-convolution}. The estimate \eqref{eq:Xr-algebra}
follows immediately by writing
\[
|f(n_1)|\le \|f\|_{X_r}w_r(n_1),
\qquad
|g(n_2)|\le \|g\|_{X_r}w_r(n_2),
\]
and applying \eqref{eq:weighted-convolution}.
\end{proof}

We next introduce the tree index measuring the number of nonlinear branchings.

\begin{definition}
For every $\gamma\in\bigcup_{k\ge0}\mathfrak T(k)$, define recursively
\[
\iota(\gamma):=
\begin{cases}
0, & \gamma=\mathtt c\ \text{or}\ \gamma=\mathtt d,\\[1mm]
1+\iota(\gamma_1)+\iota(\gamma_2), & \gamma=(\gamma_1,\gamma_2).
\end{cases}
\]
\end{definition}
We have the following relations among tree indices
\begin{lemma}\label{lem:tree-index-relations}
For every $\gamma\in\bigcup_{k\ge0}\mathfrak T(k)$, one has
\begin{equation}\label{eq:sigma-iota}
\sigma(\gamma)=\iota(\gamma)+1
\end{equation}
and
\begin{equation}\label{eq:iota-ell}
\iota(\gamma)\le \ell(\gamma).
\end{equation}
\end{lemma}

\begin{proof}
We argue by induction on the tree structure.

If $\gamma=\mathtt c$ or $\gamma=\mathtt d$, then
\[
\sigma(\gamma)=1,
\qquad
\iota(\gamma)=0,
\]
so \eqref{eq:sigma-iota} holds. Moreover,
\[
\iota(\mathtt c)=0\le \ell(\mathtt c)=0,
\qquad
\iota(\mathtt d)=0\le \ell(\mathtt d)=1,
\]
so \eqref{eq:iota-ell} also holds in the leaf case.

Now let $\gamma=(\gamma_1,\gamma_2)$ and assume that both identities hold for
$\gamma_1$ and $\gamma_2$. Then
\[
\sigma(\gamma)
=
\sigma(\gamma_1)+\sigma(\gamma_2)
=
(\iota(\gamma_1)+1)+(\iota(\gamma_2)+1)
=
\iota(\gamma)+1,
\]
which proves \eqref{eq:sigma-iota}. Similarly,
\[
\iota(\gamma)
=
1+\iota(\gamma_1)+\iota(\gamma_2)
\le
1+\ell(\gamma_1)+\ell(\gamma_2)
=
\ell(\gamma),
\]
which proves \eqref{eq:iota-ell}.
\end{proof}

The first polynomial estimate concerns the coefficient part of the tree expansion.

\begin{lemma}\label{lem:C-polynomial-bound-rigorous}
For every $k\ge0$, every $\gamma\in\mathfrak T(k)$, and every
$m\in N(k,\gamma)$, if
\[
\mathrm{Fl}_{k,\gamma}(m)=(m_1,\dots,m_{\sigma(\gamma)}),
\]
then
\begin{equation}\label{eq:C-polynomial-bound-rigorous}
|C(k,\gamma)(m)|
\le
A^{\sigma(\gamma)}
\prod_{j=1}^{\sigma(\gamma)}(1+|m_j|)^{-r}.
\end{equation}
\end{lemma}

\begin{proof}
By Lemma~\ref{lem:C-factorization},
\[
C(k,\gamma)(m)
=
\prod_{j=1}^{\sigma(\gamma)}
a_{\eta_{k,\gamma,j}}(m_j).
\]
Since each $\eta_{k,\gamma,j}$ is either $\mathtt c$ or $\mathtt d$, the assumption
\eqref{eq:poly-initial-cd} implies
\[
|a_{\eta_{k,\gamma,j}}(m_j)|
\le
A(1+|m_j|)^{-r},
\qquad
1\le j\le \sigma(\gamma).
\]
Therefore,
\[
|C(k,\gamma)(m)|
\le
A^{\sigma(\gamma)}
\prod_{j=1}^{\sigma(\gamma)}(1+|m_j|)^{-r}.
\]
This proves \eqref{eq:C-polynomial-bound-rigorous}.
\end{proof}

The next lemma is the key replacement for the compressed coefficient.

\begin{lemma}\label{lem:sum-C-polynomial-bound}
For every $k\ge0$, every $\gamma\in\mathfrak T(k)$, and every $n\in\mathbb Z^\nu$,
one has
\begin{equation}\label{eq:sum-C-polynomial-bound}
\sum_{\substack{m\in N(k,\gamma)\\ \mu(m)=n}}
|C(k,\gamma)(m)|
\le
A\,(A K_{r,\nu})^{\iota(\gamma)}(1+|n|)^{-r}.
\end{equation}
\end{lemma}

\begin{proof}
We argue by induction on the tree structure.

If $\gamma=\mathtt c$, then
\[
\sum_{\substack{m\in N(k,\mathtt c)\\ \mu(m)=n}}
|C(k,\mathtt c)(m)|
=
|c(n)|
\le
A(1+|n|)^{-r}.
\]
Since $\iota(\mathtt c)=0$, this is exactly \eqref{eq:sum-C-polynomial-bound}.
The case $\gamma=\mathtt d$ is identical.

Now let $k\ge1$ and let $\gamma=(\gamma_1,\gamma_2)\in(\mathfrak T(k-1))^2$.
Using \eqref{eq:recursive-C} and \eqref{eq:mu-definition}, we obtain
\begin{align*}
&\sum_{\substack{m\in N(k,\gamma)\\ \mu(m)=n}}
|C(k,\gamma)(m)| \\
&\qquad=
\sum_{\substack{n_1,n_2\in\mathbb Z^\nu\\ n_1+n_2=n}}
\left(
\sum_{\substack{m_1\in N(k-1,\gamma_1)\\ \mu(m_1)=n_1}}
|C(k-1,\gamma_1)(m_1)|
\right)
\left(
\sum_{\substack{m_2\in N(k-1,\gamma_2)\\ \mu(m_2)=n_2}}
|C(k-1,\gamma_2)(m_2)|
\right).
\end{align*}
Applying the induction hypothesis to $\gamma_1$ and $\gamma_2$, we get
\begin{align*}
&\sum_{\substack{m\in N(k,\gamma)\\ \mu(m)=n}}
|C(k,\gamma)(m)| \\
&\qquad\le
A(AK_{r,\nu})^{\iota(\gamma_1)}
A(AK_{r,\nu})^{\iota(\gamma_2)}
\sum_{\substack{n_1,n_2\in\mathbb Z^\nu\\ n_1+n_2=n}}
w_r(n_1)w_r(n_2).
\end{align*}
Now \eqref{eq:weighted-convolution} yields
\begin{align*}
\sum_{\substack{m\in N(k,\gamma)\\ \mu(m)=n}}
|C(k,\gamma)(m)|
&\le
A^2 (AK_{r,\nu})^{\iota(\gamma_1)+\iota(\gamma_2)}
K_{r,\nu}\,w_r(n) \\
&=
A(AK_{r,\nu})^{1+\iota(\gamma_1)+\iota(\gamma_2)}w_r(n) \\
&=
A(AK_{r,\nu})^{\iota(\gamma)}(1+|n|)^{-r}.
\end{align*}
This proves \eqref{eq:sum-C-polynomial-bound}.
\end{proof}

\subsection{Uniform polynomial decay of the Picard sequence}

We now prove a uniform polynomial decay estimate for the Picard sequence.

\begin{lemma}\label{lem:uniform-polynomial-decay}
Set
\begin{equation}\label{eq:M-B-L-polynomial}
M_r:=\max\{1,AK_{r,\nu}\},
\qquad
L_r:=\frac{1}{5M_r}.
\end{equation}
Then, for every $k\ge0$, every $0\le t\le L_r$, and every $n\in\mathbb Z^\nu$,
\begin{equation}\label{eq:uniform-polynomial-decay}
|c_k(t,n)|\le 2A(1+|n|)^{-r}.
\end{equation}
\end{lemma}

\begin{proof}
By Lemma~\ref{lem:tree-representation},
\[
c_k(t,n)
=
\sum_{\gamma\in\mathfrak T(k)}
\sum_{\substack{m\in N(k,\gamma)\\ \mu(m)=n}}
C(k,\gamma)(m)\,J(k,\gamma)(t,m).
\]
Hence,
\[
|c_k(t,n)|
\le
\sum_{\gamma\in\mathfrak T(k)}
\sum_{\substack{m\in N(k,\gamma)\\ \mu(m)=n}}
|C(k,\gamma)(m)|\,|J(k,\gamma)(t,m)|.
\]
Using Lemma~\ref{lem:CJ-estimates}, namely
\[
|J(k,\gamma)(t,m)|\le \frac{t^{\ell(\gamma)}}{D(\gamma)},
\]
we obtain
\[
|c_k(t,n)|
\le
\sum_{\gamma\in\mathfrak T(k)}
\frac{t^{\ell(\gamma)}}{D(\gamma)}
\sum_{\substack{m\in N(k,\gamma)\\ \mu(m)=n}}
|C(k,\gamma)(m)|.
\]
Now apply Lemma~\ref{lem:sum-C-polynomial-bound}:
\[
|c_k(t,n)|
\le
A(1+|n|)^{-r}
\sum_{\gamma\in\mathfrak T(k)}
\frac{(AK_{r,\nu})^{\iota(\gamma)}t^{\ell(\gamma)}}{D(\gamma)}.
\]
Since $M_r=\max\{1,AK_{r,\nu}\}$ and $\iota(\gamma)\le \ell(\gamma)$ by
Lemma~\ref{lem:tree-index-relations}, we have
\[
(AK_{r,\nu})^{\iota(\gamma)}t^{\ell(\gamma)}
\le
(M_r t)^{\ell(\gamma)}.
\]
Therefore,
\[
|c_k(t,n)|
\le
A(1+|n|)^{-r}
\sum_{\gamma\in\mathfrak T(k)}
\frac{(M_r t)^{\ell(\gamma)}}{D(\gamma)}.
\]
Since $0\le t\le L_r=1/(5M_r)$, we have $M_r t\le 1/5$. Hence, by the tree-sum
estimate \eqref{eq:tree-sum} from Subsection~\ref{sec:uniform-decay},
\[
\sum_{\gamma\in\mathfrak T(k)}
\frac{(M_r t)^{\ell(\gamma)}}{D(\gamma)}
\le 2.
\]
Thus
\[
|c_k(t,n)|
\le
2A(1+|n|)^{-r},
\]
which proves \eqref{eq:uniform-polynomial-decay}.
\end{proof}

As a corollary, we obtain the Cauchy estimate in $X_r$ directly, without introducing
an auxiliary supremum quantity.

\begin{corollary}\label{lem:cauchy-estimate-polynomial}
For every $k\ge1$, every $0\le t\le L_r$, and every $n\in\mathbb Z^\nu$, one has
\begin{equation}\label{eq:cauchy-estimate-polynomial}
|c_k(t,n)-c_{k-1}(t,n)|
\le
A\,
\frac{(4AK_{r,\nu} t)^k}{k!}\,
(1+|n|)^{-r}.
\end{equation}
In particular, $\{c_k\}_{k\ge0}$ is a Cauchy sequence in $C([0,L_r];X_r)$.
\end{corollary}

\begin{proof}
We argue by induction on $k$.

For $k=1$, by \eqref{eq:picard-iteration-recall}, \eqref{eq:Phi-definition}, and
Lemma~\ref{lem:uniform-polynomial-decay},
\begin{align*}
|c_1(t,n)-c_0(t,n)|
&=
\left|
\int_0^t \Phi_n(t-\tau)
\sum_{\substack{n_1,n_2\in\mathbb Z^\nu\\ n_1+n_2=n}}
c_0(\tau,n_1)c_0(\tau,n_2)\,d\tau
\right| \\
&\le
\int_0^t
\sum_{\substack{n_1,n_2\in\mathbb Z^\nu\\ n_1+n_2=n}}
|c_0(\tau,n_1)|\,|c_0(\tau,n_2)|\,d\tau \\
&\le
4A^2 \int_0^t
\sum_{\substack{n_1,n_2\in\mathbb Z^\nu\\ n_1+n_2=n}}
w_r(n_1)w_r(n_2)\,d\tau \\
&\le
4A^2K_{r,\nu}\,t\,w_r(n) \\
&=
4A^2K_{r,\nu}t\,(1+|n|)^{-r}.
\end{align*}
Thus \eqref{eq:cauchy-estimate-polynomial} holds for $k=1$.

Now let $k\ge2$ and assume that \eqref{eq:cauchy-estimate-polynomial} holds for $k-1$.
Using \eqref{eq:picard-iteration-recall}, we obtain
\begin{align*}
&|c_k(t,n)-c_{k-1}(t,n)| \\
&\qquad=
\left|
\int_0^t \Phi_n(t-\tau)
\sum_{\substack{n_1,n_2\in\mathbb Z^\nu\\ n_1+n_2=n}}
\Big(
c_{k-1}(\tau,n_1)c_{k-1}(\tau,n_2)
-
c_{k-2}(\tau,n_1)c_{k-2}(\tau,n_2)
\Big)\,d\tau
\right| \\
&\qquad\le
\int_0^t
\sum_{\substack{n_1,n_2\in\mathbb Z^\nu\\ n_1+n_2=n}}
\big|
c_{k-1}(\tau,n_1)c_{k-1}(\tau,n_2)
-
c_{k-2}(\tau,n_1)c_{k-2}(\tau,n_2)
\big|\,d\tau.
\end{align*}
Using
\[
ab-a'b'=(a-a')b+a'(b-b'),
\]
together with Lemma~\ref{lem:uniform-polynomial-decay}, we get
\begin{align*}
|c_k(t,n)-c_{k-1}(t,n)|
&\le
\int_0^t
\sum_{\substack{n_1,n_2\in\mathbb Z^\nu\\ n_1+n_2=n}}
|c_{k-1}(\tau,n_1)-c_{k-2}(\tau,n_1)|\,|c_{k-1}(\tau,n_2)|\,d\tau \\
&\quad+
\int_0^t
\sum_{\substack{n_1,n_2\in\mathbb Z^\nu\\ n_1+n_2=n}}
|c_{k-2}(\tau,n_1)|\,|c_{k-1}(\tau,n_2)-c_{k-2}(\tau,n_2)|\,d\tau \\
&\le
4A\int_0^t
\sum_{\substack{n_1,n_2\in\mathbb Z^\nu\\ n_1+n_2=n}}
|c_{k-1}(\tau,n_1)-c_{k-2}(\tau,n_1)|\,w_r(n_2)\,d\tau.
\end{align*}
Applying the induction hypothesis to the difference term and then
\eqref{eq:weighted-convolution}, we obtain
\begin{align*}
|c_k(t,n)-c_{k-1}(t,n)|
&\le
4A\int_0^t
\sum_{\substack{n_1,n_2\in\mathbb Z^\nu\\ n_1+n_2=n}}
A\,
\frac{(4AK_{r,\nu}\tau)^{k-1}}{(k-1)!}\,
w_r(n_1)w_r(n_2)\,d\tau \\
&\le
4K_{r,\nu}A^2
\int_0^t
\frac{(4AK_{r,\nu}\tau)^{k-1}}{(k-1)!}\,d\tau\,
w_r(n) \\
&=
A\,
\frac{(4AK_{r,\nu} t)^k}{k!}\,
(1+|n|)^{-r}.
\end{align*}
This proves \eqref{eq:cauchy-estimate-polynomial} for all $k\ge1$.

Finally,
\[
\sup_{0\le t\le L_r}\|c_k(t,\cdot)-c_{k-1}(t,\cdot)\|_{X_r}
\le
A\frac{(4K_{r,\nu}AL_r)^k}{k!}.
\]
Since
\[
4K_{r,\nu}AL_r
=
\frac{4AK_{r,\nu}}{5M_r}
\le
\frac45,
\]
the series of successive differences converges, and hence $\{c_k\}_{k\ge0}$ is
Cauchy in $C([0,L_r];X_r)$.
\end{proof}

\subsection{Proof of Theorem \ref{thm:polynomial decay}}

\begin{proof}[Proof of Theorem~\ref{thm:polynomial decay}]
By Corollary~\ref{lem:cauchy-estimate-polynomial}, there exists
\[
c\in C([0,L_r];X_r)
\]
such that
\[
\lim_{k\to\infty}\sup_{0\le t\le L_r}\|c_k(t,\cdot)-c(t,\cdot)\|_{X_r}=0.
\]
Moreover, Lemma~\ref{lem:uniform-polynomial-decay} implies
\[
|c(t,n)|\le 2A(1+|n|)^{-r},
\qquad
0\le t\le L_r,\ \ n\in\mathbb Z^\nu.
\]

We claim that $c$ solves the integral equation \eqref{eq:integral-equation}. Indeed,
the linear part passes to the limit trivially, and for the nonlinear term we use
\eqref{eq:Xr-algebra}:
\begin{align*}
&\left\|
\sum_{\substack{n_1,n_2\in\mathbb Z^\nu\\ n_1+n_2=\cdot}}
c_k(t,n_1)c_k(t,n_2)
-
\sum_{\substack{n_1,n_2\in\mathbb Z^\nu\\ n_1+n_2=\cdot}}
c(t,n_1)c(t,n_2)
\right\|_{X_r} \\
&\qquad\le
K_{r,\nu}\|c_k(t,\cdot)-c(t,\cdot)\|_{X_r}\|c_k(t,\cdot)\|_{X_r}
+
K_{r,\nu}\|c(t,\cdot)\|_{X_r}\|c_k(t,\cdot)-c(t,\cdot)\|_{X_r}.
\end{align*}
Since $\sup_k\|c_k(t,\cdot)\|_{X_r}\le 2A$, the right-hand side tends to zero
uniformly in $t$. Hence we may pass to the limit in \eqref{eq:picard-iteration-recall}
and conclude that
\[
c(t,n)
=
c_0(t,n)
-\beta(n)\int_0^t K_n(t-\tau)
\sum_{\substack{n_1,n_2\in\mathbb Z^\nu\\ n_1+n_2=n}}
c(\tau,n_1)c(\tau,n_2)\,d\tau,
\qquad n\in\mathbb Z^\nu.
\]

Define
\[
u(t,x):=\sum_{n\in\mathbb Z^\nu}c(t,n)e^{i\langle n,\omega\rangle x}.
\]
Since $r>\nu+2$, the series defining $u$, $u_t$, $u_{tt}$, $u_{xx}$, and $u_{xxtt}$
are absolutely and uniformly convergent on $[0,L_r]\times\mathbb R$; the same is true
for $(u^2)_{xx}$ by the weighted convolution estimate. Therefore termwise
differentiation is justified, and $u$ is a classical spatially quasi-periodic solution
to \eqref{eq:IBq-intro}. The bound
\[
|\widehat u(t,n)|=|c(t,n)|\le 2A(1+|n|)^{-r}
\]
has already been proved.

It remains to prove uniqueness. Let $c,\widetilde c\in C([0,L_r];X_r)$ be two
solutions of \eqref{eq:integral-equation} with the same initial data and satisfying
\[
\|c(t,\cdot)\|_{X_r}\le 2A,
\qquad
\|\widetilde c(t,\cdot)\|_{X_r}\le 2A,
\qquad 0\le t\le L_r.
\]
Subtracting the two integral equations and using \eqref{eq:Xr-algebra}, we get
\[
\|c(t,\cdot)-\widetilde c(t,\cdot)\|_{X_r}
\le
4AK_{r,\nu}\int_0^t
\|c(\tau,\cdot)-\widetilde c(\tau,\cdot)\|_{X_r}\,d\tau.
\]
By Gronwall's inequality,
\[
\|c(t,\cdot)-\widetilde c(t,\cdot)\|_{X_r}=0,
\qquad 0\le t\le L_r.
\]
Hence $c=\widetilde c$, and therefore the corresponding quasi-periodic solution is
unique in the class satisfying the polynomial decay estimate.

This completes the proof of Theorem~\ref{thm:polynomial decay}.
\end{proof}

\section{Integer power nonlinearities}\label{sec:general-power}

In this final section we explain how the preceding arguments extend to the
generalized improved Boussinesq equation
\begin{equation}\label{eq:gp-main}
    u_{tt}-u_{xx}-u_{xxtt}-(u^p)_{xx}=0,
    \qquad (t,x)\in\mathbb R\times\mathbb R,
\end{equation}
where
$
    p\in\mathbb N$ with $p\ge3.$
The restriction that $p$ is an integer is essential for the present argument, since
then the Fourier coefficients of $u^p$ are given by a finite $p$-fold Cauchy product.

We use the notation introduced in Subsection~\ref{sec:reduction}. In particular,
$\theta_n$, $\beta(n)$, $\Omega(n)$, $G_n$, $K_n$, $c_0(t,n)$, and $\Phi_n(t)$ have
the same meaning as before. For a sequence $a:\mathbb Z^\nu\to\mathbb C$, we write
\[
a^{*p}(n)
:=
\sum_{\substack{n_1,\ldots,n_p\in\mathbb Z^\nu\\
n_1+\cdots+n_p=n}}
\prod_{j=1}^{p}a(n_j),
\qquad n\in\mathbb Z^\nu .
\]
Then the same Fourier-side reduction as in Subsection~\ref{sec:reduction} gives
\begin{equation}\label{eq:gp-ode}
c_{tt}(t,n)+\Omega(n)^2c(t,n)
=
-\Omega(n)^2c(t,\cdot)^{*p}(n),
\qquad n\in\mathbb Z^\nu .
\end{equation}
Equivalently,
\begin{equation}\label{eq:gp-integral}
c(t,n)
=
c_0(t,n)
-
\beta(n)\int_0^t K_n(t-\tau)c(\tau,\cdot)^{*p}(n)\,d\tau .
\end{equation}
Thus the only change from \eqref{eq:integral-equation} is that the quadratic
convolution is replaced by the $p$-fold convolution.

\subsection{Change of the tree structure}\label{subsec:p-ary-tree}

We briefly describe how the tree structure changes in the integer power case. The
set of leaves is unchanged:
\[
    \mathfrak L=\{\mathtt c,\mathtt d\},
\]
corresponding respectively to the initial position and initial velocity. The only
structural change is that every nonlinear vertex now has $p$ children rather than
two. Hence the binary tree family is replaced by the $p$-ary family
\[
\mathfrak T_p(0):=\mathfrak L,
\qquad
\mathfrak T_p(k):=\mathfrak L\cup\bigl(\mathfrak T_p(k-1)\bigr)^p,
\qquad k\ge1.
\]

The Picard sequence associated with \eqref{eq:gp-integral} is
\[
c^{(p)}_0(t,n):=c_0(t,n),
\]
and, for $k\ge1$,
\begin{equation}\label{eq:gp-picard}
c^{(p)}_k(t,n)
=
c_0(t,n)
-
\int_0^t\Phi_n(t-\tau)
\sum_{\substack{n_1,\ldots,n_p\in\mathbb Z^\nu\\
n_1+\cdots+n_p=n}}
\prod_{j=1}^p c^{(p)}_{k-1}(\tau,n_j)\,d\tau .
\end{equation}

For a non-leaf tree
\[
\gamma=(\gamma_1,\ldots,\gamma_p)\in\bigl(\mathfrak T_p(k-1)\bigr)^p,
\]
we define the basic indices recursively by
\[
\sigma_p(\gamma)=\sum_{j=1}^p\sigma_p(\gamma_j),
\qquad
\ell_p(\gamma)=1+\sum_{j=1}^p\ell_p(\gamma_j),
\]
and
\[
D_p(\gamma)=\ell_p(\gamma)\prod_{j=1}^pD_p(\gamma_j),
\]
with the same leaf values as in the quadratic case:
\[
\sigma_p(\mathtt c)=\sigma_p(\mathtt d)=1,
\qquad
\ell_p(\mathtt c)=0,\quad \ell_p(\mathtt d)=1,
\qquad
D_p(\mathtt c)=D_p(\mathtt d)=1.
\]

The frequency variables are defined as before, except that binary products are replaced
by $p$-fold products. Namely,
\[
N_p(k,\gamma)
=
\prod_{j=1}^p N_p(k-1,\gamma_j),
\qquad
\mu_p(m)=\sum_{j=1}^p\mu_p(m^{(j)}),
\]
when
\[
m=(m^{(1)},\ldots,m^{(p)})\in N_p(k,\gamma).
\]
The coefficients in the tree expansion satisfy
\[
C_p(k,\gamma)(m)
=
\prod_{j=1}^p C_p(k-1,\gamma_j)(m^{(j)}),
\]
and
\[
J_p(k,\gamma)(t,m)
=
-\int_0^t
\Phi_{\mu_p(m)}(t-\tau)
\prod_{j=1}^p
J_p(k-1,\gamma_j)(\tau,m^{(j)})\,d\tau .
\]
For the leaves, $C_p$ and $J_p$ are the same as in
\eqref{eq:base-c}--\eqref{eq:base-d}.

With these definitions, the Picard sequence has the representation
\begin{equation}\label{eq:p-tree-representation}
c^{(p)}_k(t,n)
=
\sum_{\gamma\in\mathfrak T_p(k)}
\sum_{\substack{m\in N_p(k,\gamma)\\ \mu_p(m)=n}}
C_p(k,\gamma)(m)J_p(k,\gamma)(t,m).
\end{equation}
This is the direct analogue of Lemma~\ref{lem:tree-representation}. The proof is
identical, with each quadratic branching replaced by a $p$-fold branching.

The corresponding coefficient estimates have the same form. Under the exponential
initial decay assumption,
\[
|C_p(k,\gamma)(m)|
\le
A^{\sigma_p(\gamma)}
\exp\left(-\rho\sum_{j=1}^{\sigma_p(\gamma)}|m_j|\right),
\]
after flattening the leaf variables. Moreover,
\[
|J_p(k,\gamma)(t,m)|
\le
\frac{t^{\ell_p(\gamma)}}{D_p(\gamma)}.
\]
The elementary combinatorial inequality
\[
\sigma(\gamma)\le \ell(\gamma)+1
\]
is replaced by
\[
\sigma_p(\gamma)\le (p-1)\ell_p(\gamma)+1.
\]
Thus the full tree expansion changes only by replacing binary trees with $p$-ary
trees. The estimates below can therefore be closed in the same weighted convolution
spaces as in the quadratic case.

\subsection{Exponential decay}

\begin{theorem}[Integer powers: exponential decay]\label{thm:gp-exp}
Let $p\in\mathbb N$, $p\ge3$, and let $\omega\in\mathbb R^\nu$ be non-resonant.
Assume that there exist $A>0$ and $0<\rho\le1$ such that
\[
|c(n)|\le A e^{-\rho|n|},
\qquad
|d(n)|\le A e^{-\rho|n|},
\qquad n\in\mathbb Z^\nu .
\]
Let
\[
b_\rho:=(6\rho^{-1})^\nu,
\qquad
M_\rho:=\max\{1,Ab_\rho\},
\qquad
B_\rho:=2M_\rho,
\]
and set
\[
L_{p,\rho}:=
\frac{1}{p\,2^{p+3}M_\rho^{p-1}} .
\]
Then, on $[0,L_{p,\rho}]$, the quasi-periodic Cauchy problem for
\eqref{eq:gp-main} admits a classical spatially quasi-periodic solution
\[
u(t,x)=\sum_{n\in\mathbb Z^\nu}\widehat u(t,n)e^{i\langle n,\omega\rangle x}
\]
satisfying
\[
|\widehat u(t,n)|
\le
B_\rho e^{-\frac{\rho}{2}|n|},
\qquad
0\le t\le L_{p,\rho},\quad n\in\mathbb Z^\nu .
\]
Moreover, this solution is unique among all spatially quasi-periodic solutions
satisfying the above exponential decay estimate.
\end{theorem}

\begin{proof}[Sketch of proof]
Let
\[
Y_s
:=
\left\{
a:\mathbb Z^\nu\to\mathbb C:
\|a\|_{Y_s}:=\sum_{n\in\mathbb Z^\nu}e^{s|n|}|a(n)|<\infty
\right\}.
\]
The space $Y_s$ is a Banach algebra under discrete convolution. Since
\[
\sum_{n\in\mathbb Z^\nu}e^{-\frac{\rho}{2}|n|}
\le b_\rho,
\]
the initial data belong to $Y_{\rho/2}$ with norm controlled by $Ab_\rho$.

Define the fixed-point map associated with \eqref{eq:gp-integral} by
\[
(\mathcal T_p a)(t,n)
:=
c_0(t,n)
-
\beta(n)\int_0^t K_n(t-\tau)a(\tau,\cdot)^{*p}(n)\,d\tau .
\]
Using the bounds in \eqref{eq:basic-propagator-bounds}, we obtain
\[
\|\mathcal T_p a\|_{C([0,T];Y_{\rho/2})}
\le
(1+T)Ab_\rho
+
T\|a\|_{C([0,T];Y_{\rho/2})}^{p}.
\]
On the closed ball
\[
\mathcal B_T
:=
\left\{
a\in C([0,T];Y_{\rho/2}):
\|a\|_{C([0,T];Y_{\rho/2})}\le2M_\rho
\right\},
\]
and with $T=L_{p,\rho}$, this gives
\[
\|\mathcal T_p a\|_{C([0,T];Y_{\rho/2})}
\le
(1+T)M_\rho+T(2M_\rho)^p
<2M_\rho.
\]
Moreover, for $a,b\in\mathcal B_T$,
\[
\|a^{*p}-b^{*p}\|_{Y_{\rho/2}}
\le
p(2M_\rho)^{p-1}\|a-b\|_{Y_{\rho/2}},
\]
and hence
\[
\|\mathcal T_p a-\mathcal T_p b\|_{C([0,T];Y_{\rho/2})}
\le
Tp(2M_\rho)^{p-1}
\|a-b\|_{C([0,T];Y_{\rho/2})}
\le
\frac1{16}
\|a-b\|_{C([0,T];Y_{\rho/2})}.
\]
Thus $\mathcal T_p$ has a unique fixed point in $\mathcal B_T$. The pointwise
exponential estimate follows immediately from the $Y_{\rho/2}$ bound.

The passage from the coefficient equation to a classical solution is justified as
in the proof of Theorem~\ref{thm:exponential decay}. Indeed, exponential decay
gives the absolute and uniform convergence of all differentiated Fourier series
appearing in \eqref{eq:gp-main}. For uniqueness in the pointwise exponential class,
one argues in the weaker space $Y_{\rho/4}$. The bound
\[
|\widehat u(t,n)|\le B_\rho e^{-\frac{\rho}{2}|n|}
\]
implies uniform boundedness in $Y_{\rho/4}$, and the same multilinear estimate,
followed by Gronwall's inequality, gives uniqueness.
\end{proof}

\subsection{Polynomial decay}

We now turn to polynomially decaying Fourier coefficients. Recall the notation
$w_r$, $X_r$, $H(r;\nu)$, and $K_{r,\nu}$ from the proof of
Theorem~\ref{thm:polynomial decay}. By iterating \eqref{eq:Xr-algebra}, one obtains
the $p$-fold estimate
\begin{equation}\label{eq:p-fold-Xr-algebra}
\|a_1*\cdots*a_p\|_{X_r}
\le
K_{r,\nu}^{p-1}
\prod_{j=1}^p\|a_j\|_{X_r},
\qquad a_1,\ldots,a_p\in X_r .
\end{equation}

\begin{theorem}[Integer powers: polynomial decay]\label{thm:gp-poly}
Let $p\in\mathbb N$, $p\ge3$, and let $\omega\in\mathbb R^\nu$ be non-resonant.
Assume that there exist $A>0$ and $r>\nu+2$ such that
\[
|c(n)|\le A(1+|n|)^{-r},
\qquad
|d(n)|\le A(1+|n|)^{-r},
\qquad n\in\mathbb Z^\nu .
\]
Set
\[
M_{p,r}:=\max\{1,AK_{r,\nu}\},
\qquad
L_{p,r}:=
\frac{1}{p\,2^{p+3}M_{p,r}^{p-1}} .
\]
Then, on $[0,L_{p,r}]$, the quasi-periodic Cauchy problem for \eqref{eq:gp-main}
admits a classical spatially quasi-periodic solution
\[
u(t,x)=\sum_{n\in\mathbb Z^\nu}\widehat u(t,n)e^{i\langle n,\omega\rangle x}
\]
satisfying
\[
|\widehat u(t,n)|
\le
2A(1+|n|)^{-r},
\qquad
0\le t\le L_{p,r},\quad n\in\mathbb Z^\nu .
\]
Moreover, this solution is unique among all spatially quasi-periodic solutions
satisfying the above polynomial decay estimate.
\end{theorem}

\begin{proof}[Sketch of proof]
We solve \eqref{eq:gp-integral} in $C([0,T];X_r)$. Let
\[
\mathcal B_T^r
:=
\left\{
a\in C([0,T];X_r):
\|a\|_{C([0,T];X_r)}\le2A
\right\}.
\]
Using \eqref{eq:p-fold-Xr-algebra} and the bounds in
\eqref{eq:basic-propagator-bounds}, we obtain
\[
\|\mathcal T_p a\|_{C([0,T];X_r)}
\le
(1+T)A
+
T K_{r,\nu}^{p-1}(2A)^p .
\]
For $T=L_{p,r}$, the last term satisfies
\[
T K_{r,\nu}^{p-1}(2A)^p
=
A\,T\,2^p(AK_{r,\nu})^{p-1}
\le
\frac{A}{8p}.
\]
Hence $\mathcal T_p$ maps $\mathcal B_T^r$ into itself. Similarly,
\[
\|\mathcal T_p a-\mathcal T_p b\|_{C([0,T];X_r)}
\le
TpK_{r,\nu}^{p-1}(2A)^{p-1}
\|a-b\|_{C([0,T];X_r)}
\le
\frac1{16}
\|a-b\|_{C([0,T];X_r)}.
\]
Thus $\mathcal T_p$ has a unique fixed point
\[
c\in C([0,L_{p,r}];X_r),
\]
and
\[
|c(t,n)|\le2A(1+|n|)^{-r}.
\]

It remains to justify that this fixed point yields a classical solution. Since
$c\in C([0,L_{p,r}];X_r)$ and $X_r$ is stable under the $p$-fold convolution by
\eqref{eq:p-fold-Xr-algebra}, we have
\[
c(\cdot,\cdot)^{*p}\in C([0,L_{p,r}];X_r).
\]
Differentiating \eqref{eq:gp-integral} in time gives
\[
c_t\in C([0,L_{p,r}];X_r),
\]
and the coefficient equation \eqref{eq:gp-ode} gives
\[
c_{tt}(t,n)
=
-\beta(n)c(t,n)-\beta(n)c(t,\cdot)^{*p}(n),
\]
hence
\[
c_{tt}\in C([0,L_{p,r}];X_r).
\]
Since $|\langle n,\omega\rangle|\le C_\omega |n|$ and $r>\nu+2$, all Fourier
series obtained from $u$, $u_t$, $u_{tt}$, $u_{xx}$, $u_{xxtt}$, and $(u^p)_{xx}$
by termwise differentiation converge absolutely and uniformly on
$[0,L_{p,r}]\times\mathbb R$. Therefore the resulting function is a classical
spatially quasi-periodic solution of \eqref{eq:gp-main}.

Finally, if $c$ and $\widetilde c$ are two solutions satisfying the stated polynomial
decay bound, then by \eqref{eq:p-fold-Xr-algebra},
\[
\|c(t,\cdot)-\widetilde c(t,\cdot)\|_{X_r}
\le
pK_{r,\nu}^{p-1}(2A)^{p-1}
\int_0^t
\|c(\tau,\cdot)-\widetilde c(\tau,\cdot)\|_{X_r}\,d\tau .
\]
Gronwall's inequality gives $c=\widetilde c$ on $[0,L_{p,r}]$.
\end{proof}

\section*{Acknowledgement}
We are very grateful to Fei Xu for many useful discussions. H.Z. also thanks Professors Wen Huang, Shiping Liu, and Zhiyan Zhao for helpful discussions. H.Z. is supported by the National Key R \& D Program of China 2023YFA1010200 and the
 National Natural Science Foundation of China No. 12431004. Z.W. is supported by the National Natural Science Foundation of China No. 12341102.

\subsection*{Availability of Data}  No data was used for the research described in the article.

\subsection*{Declarations of Conflict of Interest}
The authors do not have any possible conflicts of
interest.

\end{document}